\documentclass[11pt,reqno,tbtags]{amsart}
\usepackage{amssymb}
\usepackage{natbib}
\bibpunct[, ]{[}{]}{;}{n}{,}{,}

\numberwithin{equation}{section}
\allowdisplaybreaks

\newtheorem{theorem}{Theorem}[section]
\newtheorem{lemma}[theorem]{Lemma}

\newtheorem{corollary}[theorem]{Corollary}

\theoremstyle{definition}

\newtheorem{definition}[theorem]{Definition}
\newtheorem{problem}[theorem]{Problem}
\newtheorem{remark}[theorem]{Remark}

\newtheorem*{ack}{Acknowledgement}

\theoremstyle{remark}

\newenvironment{romenumerate}{\begin{enumerate}
 }{\end{enumerate}}

\newcounter{oldenumi}
{\setcounter{oldenumi}{\value{enumi}}
\begin{romenumerate} \setcounter{enumi}{\value{oldenumi}}}
{\end{romenumerate}}

\newcounter{thmenumerate}

\newcounter{xenumerate}   

\newcommand\pfitem[1]{\par(#1):}

\newcommand{\refT}[1]{Theorem~\ref{#1}}
\newcommand{\refC}[1]{Corollary~\ref{#1}}
\newcommand{\refL}[1]{Lemma~\ref{#1}}
\newcommand{\refR}[1]{Remark~\ref{#1}}
\newcommand{\refS}[1]{Section~\ref{#1}}
\newcommand{\refSS}[1]{Subsection~\ref{#1}}
\newcommand{\refD}[1]{Definition~\ref{#1}}
\newcommand{\refP}[1]{Problem~\ref{#1}}

\newcommand{\refApp}[1]{Appendix~\ref{#1}}

\newcommand{\refand}[2]{\ref{#1} and~\ref{#2}}

\begingroup
  \divide\count255 by 60
  \multiply\count255 by -60
  \advance\count255 by \time
  \ifnum \count255 < 10 \xdef\klockan{\the\count1.0\the\count255}
  \else\xdef\klockan{\the\count1.\the\count255}\fi
\endgroup

\newcommand\nopf{\qed}   

\newcommand\set[1]{\ensuremath{\{#1\}}}

\newcommand\Bigset[1]{\ensuremath{\Bigl\{#1\Bigr\}}}

\newcommand\xpar[1]{(#1)}
\newcommand\bigpar[1]{\bigl(#1\bigr)}
\newcommand\Bigpar[1]{\Bigl(#1\Bigr)}

\newcommand\lrpar[1]{\left(#1\right)}

\newcommand\bigabs[1]{\bigl|#1\bigr|}

\newcommand\biggabs[1]{\biggl|#1\biggr|}
\newcommand\lrabs[1]{\left|#1\right|}
\def\rompar(#1){\textup(#1\textup)}    
\newcommand\xfrac[2]{#1/#2}

\newcommand\parfrac[2]{\lrpar{\frac{#1}{#2}}}

\def\xexp(#1){e^{#1}}
\newcommand\ceil[1]{\lceil#1\rceil}
\newcommand\floor[1]{\lfloor#1\rfloor}

\newcommand\ntoo{\ensuremath{{n\to\infty}}}

\newcommand\norm[1]{\|#1\|}

\newcommand\ie{i.e.\spacefactor=1000}
\newcommand\eg{e.g.\spacefactor=1000}

\newcommand\cf{cf.\spacefactor=1000}

\newcommand{\aex}{a.e.\spacefactor=1000}

\newcommand{\tend}{\longrightarrow}
\newcommand\dto{\overset{\mathrm{d}}{\tend}}
\newcommand\pto{\overset{\mathrm{p}}{\tend}}

\newcommand\lto{\overset{L^1}{\tend}}

\newcommand\bbR{\mathbb R}

\newcounter{CC}
\newcounter{cc}

\newcommand\E{\operatorname{\mathbb E{}}}
\renewcommand\P{\operatorname{\mathbb P{}}}
\newcommand\Var{\operatorname{Var}}

\newcommand\ga{\alpha}
\newcommand\gb{\beta}
\newcommand\gd{\delta}

\newcommand\gf{\varphi}
\newcommand\gam{\gamma}

\newcommand\gl{\lambda}

\newcommand\gs{\sigma}

\newcommand\eps{\varepsilon}

\newcommand\cI{\mathcal I}
\newcommand\cJ{\mathcal J}

\newcommand\cP{\mathcal P}

\newcommand\cS{{\mathcal S}}

\newcommand\tE{{\tilde E}}
\newcommand\tf{{\tilde f}}

\newcommand\ett[1]{\boldsymbol1[#1]} 
\newcommand\etta{\boldsymbol1} 
\def\[#1]{[\![#1]\!]}

\newcommand\qw{^{-1}}
\newcommand\qww{^{-2}}

\renewcommand{\=}{:=}

\newcommand\intoi{\int_0^1}

\newcommand\oi{[0,1]}
\newcommand\oio{(0,1)}

\newcommand\setoi{\set{0,1}}

\newcommand\dd{\,\textup{d}}

\newcommand\lhs{left-hand side}
\newcommand\rhs{right-hand side}

\newcommand\aaf{A_1,\dots,A_\ff}
\newcommand\aafx{A_1\times\dots\times A_\ff}
\newcommand\aam{A_1,\dots,A_m}
\newcommand\aamx{A_1\times\dots\times A_m}
\newcommand\aagsmx{A_{\gs(1)}\times\dots\times A_{\gs(m)}}
\newcommand\ijifx{I_{nj_1}''\times\dots\times I_{nj_\ff}''}
\newcommand\aayf{A'_1,\dots,A'_\ff}
\newcommand\aafyx{A'_1\times\dots\times A'_\ff}
\newcommand\AAm{A_2,\dots, A_m}
\newcommand\AAmx{A_2\times\dots\times A_m}
\newcommand\fAAm{f^{A_2,\dots, A_m}}

\newcommand\bbfx{B_1\times\dots\times B_\ff}
\newcommand\bbm{B_1,\dots,B_m}
\newcommand\bbmx{B_1\times\dots\times B_m}
\newcommand\bbnm{B_{n1},\dots,B_{nm}}
\newcommand\uuf{U_1,\dots,U_\ff}
\newcommand\uuif{U'_1,\dots,U'_\ff}
\newcommand\uuifx{U'_1\times\dots\times U'_\ff}
\newcommand\uuiifx{U''_1\times\dots\times U''_\ff}
\newcommand\xxf{x_1,\dots,x_\ff}
\newcommand\xxm{x_1,\dots,x_m}
\newcommand\yym{y_1,\dots,y_m}
\newcommand\iim{i_1,\dots,i_m}
\newcommand\aaiimx{A_{i_1}\times\dotsm\times A_{i_m}}
\newcommand\intaaiimx{\int_{\aaiimx}}
\newcommand\xxgsm{x_{\gs(1)},\dots,x_{\gs(m)}}
\newcommand\ff{{|F|}}
\newcommand\eff{^{e(F)}}
\newcommand\prodif{\prod_{i=1}^\ff}

\newcommand{\sumim}{\sum_{i=1}^m}
\newcommand{\sumiom}{\sum_{i=0}^m}
\newcommand{\sumif}{\sum_{i=1}^\ff}
\newcommand\intaafx{\int_{\aafx}}
\newcommand\intaamx{\int_{\aamx}}
\newcommand\intaf{\int_{A^\ff}}
\newcommand\oiff{\oi^\ff}
\newcommand\Nq{N^*}
\newcommand\Psiq{\Psi^*}
\newcommand\tPsiq{\widetilde\Psi^*}
\newcommand\fS{\mathfrak S}
\newcommand\fsf{\fS_\ff}
\newcommand\psifx[1]{\Psi_{F,#1}}
\newcommand\psifw{\psifx W}
\newcommand\psifwn{\psifx {W_n}}
\newcommand\tpsifx[1]{\widetilde\Psi_{F,#1}}
\newcommand\tpsifw{\tpsifx W}

\newcommand\psiqfx[1]{\Psiq_{F,#1}}
\newcommand\tpsiqfx[1]{\tPsiq_{F,#1}}
\newcommand\psiqfw{\psiqfx W}
\newcommand\tpsiqfw{\tpsiqfx W}

\newcommand\phiw{\Phi_W}
\newcommand\Phix{\Phi} 
\newcommand\mpx{measure preserving}
\newcommand\cut{_{\square}}
\newcommand\cn[1]{\norm{#1}\cut}
\newcommand\cnx[2]{\norm{#1}_{\square,#2}}
\newcommand\mppqr{mixed $(p,\cpp)$-quasi-random}
\newcommand\Sze{Szemer\'edi}
\newcommand\SRL{\Sze{} regularity lemma}
\newcommand\SsRL{\Sze's regularity lemma}
\newcommand\xex{X^\eps_x}
\newcommand\cp{\overline}
\newcommand\pqr{$p$-\qr}
\newcommand\cpp{\bar p}
\newcommand\cpqr{$\cpp$-\qr}
\newcommand\qr{quasi-random}
\newcommand\HF{\ensuremath{\mathsf{HI}}}
\newcommand\HFp{\HF(p)}
\newcommand\gnq{\ensuremath{(G_n)}}
\newcommand\gnabs{\ensuremath{|G_n|}}
\newcommand\iniy{I_{ni}''}
\newcommand\injy{I_{nj}''}
\newcommand\inj{I_{nj}}
\newcommand\setdiff{\bigtriangleup}
\newcommand\xim{\xi_1,\dots,\xi_m}
\newcommand\gax{\gam}
\newcommand\gc{\ga}
\newcommand\egn{e_{G_n}}
\newcommand\cpu{\cp U}
\newcommand\cpax{\oi\setminus A}
\newcommand\cpA{\cp A}
\renewcommand\wp{W_{\cP}}

\newcommand{\tinj}{t_{\mathrm{inj}}}
\newcommand{\tind}{t_{\mathrm{ind}}}

\newcommand{\dcut}{\gd_\square}

\newcommand\nn{[n]}

\newcommand\sss{\cS}

\newcommand\lp{Lebesgue point}
\newcommand\fall[2]{(#1)_{#2}}
\newcommand\fallm[1]{\fall{#1}m}
\providecommand\ij\relax   
\renewcommand\ij{_{ij}}
\newcommand\gdn{\gd_n}
\newcommand\etan{\eta_n}
\newcommand\bin{B_{ni}}
\newcommand\bjn{B_{nj}}
\newcommand\wijn{w_{ij,n}}
\newcommand\gbfp{\gb_F(p)}
\newcommand\gbf{\gb_F}

\newcommand{\Lovasz}{Lov\'asz}

\newcommand\urladdrx[1]{{\urladdr{\def~{{\tiny$\sim$}}#1}}}

\begin{document}
\title
{Quasi-random graphs and graph limits}

\date{May 20, 2009} 

\author{Svante Janson}
\thanks{Research partly done at Institut Mittag-Leffler,  Djursholm,
Sweden}

\address{Department of Mathematics, Uppsala University, PO Box 480,
SE-751~06 Uppsala, Sweden}
\email{svante.janson@math.uu.se}
\urladdrx{http://www.math.uu.se/~svante/}

\subjclass[2000]{05C99} 

\begin{abstract} 
We use the theory of graph limits to
study several quasi-random properties, mainly dealing with 
various versions of hereditary subgraph counts. 
The main idea is to
transfer the properties of (sequences of) graphs to properties of
graphons, and to show that the resulting graphon properties only can
be satisfied by constant graphons.
These \qr{} properties have been studied before by other authors,
but our approach gives proofs that we find cleaner, and which
avoid the error terms and $\eps$ in the traditional
arguments using 
the \SRL. On the other hand, other technical problems sometimes arise in
analysing the graphon properties; in
particular, a
measure-theoretic problem on elimination of null sets
that arises in this way is treated in an
appendix.
\end{abstract}

\maketitle

\section{Introduction}\label{Sintro}

A \emph{\qr} graph is a graph that 'looks like' a
random graph. Formally, this is best defined for a sequence of graphs
$(G_n)$ with $|G_n|\to\infty$. 
\citet{Thomason87a,Thomason87b} and \citet{ChungGW:quasi} showed
that a number of different 'random-like' conditions on such a sequence
are equivalent, and we say that \gnq{} is \pqr{} if
it satisfies these conditions. (Here $p\in\oi$ is a
parameter.)
We give one of these conditions, which is based on subgraph counts, 
in \eqref{pqr} below.
Other characterizations have been added by various authors.
The present paper studies in particular 
hereditarily extended subgraph
count properties found by \citet{SS:nni,SS:ind},
\citet{Shapira},  \citet{ShapiraY} and \citet{Yuster}; see \refS{Ssub}.
See also Sections \refand{Scut}{Sdeg} for further related equivalent properties
(on sizes of cuts) found by \citet{ChungGW:quasi} and \citet{ChungG}.

The theory of \emph{graph limits}
also concern the asymptotic behaviour of sequences
$(G_n)$ of graphs with $|G_n|\to\infty$. 
A notion of convergence of such sequences was introduced by
\citet{LSz} and further developed by
\citet{BCLSV1,BCLSV2}.
This may be seen as giving
the space of (unlabelled) graphs a
suitable metric; the convergent sequences are the Cauchy sequences in
this metric, and the completion of the space of unlabelled graphs in
this metric is the space of (graphs and) graph limits.
The graph limits are thus defined in a rather abstract way, but there
are also more concrete representations of them.
One important 
representation \cite{LSz,BCLSV1} uses 
a symmetric (Lebesgue) measurable function $W:\oi^2\to\oi$;
such a function is called a \emph{graphon}, and defines a unique graph
limit, see \refS{Snot} for details. Note, however, that the
representation is not unique; different graphons may be equivalent in
the sense of defining the same graph limit.
See further \cite{BCL:unique,SJ209}. 

We write, with a minor abuse of notation, $G_n\to W$, if
$\gnq$ is a sequence of graphs and $W$ is a graphon
such that $\gnq$ converges to the graph limit defined by $W$.
It is well-known that \qr{} graphs provide the simplest
example of this: \gnq{} is \pqr{} if and only if
$G_n\to p$, where $p$ is the graphon that is constant $p$
\cite{LSz}. 

A central tool to study large dense graphs is 
\SsRL, and it is not surprising that this is closely connected to
the theory of graph limits, see, \eg,
\cite{BCLSV1,LSz:Sz}.
The \SRL{} is also important for the study of \qr{}
graphs. For example, \citet{SS:Sze} gave a characterization 
of \qr{} graphs in terms of \Sze{} partitions. 
Moreover, the proofs 
in
\cite{SS:nni,SS:ind,Shapira,ShapiraY} that various properties
characterize \qr{} graphs (see \refS{Ssub}) use the \SRL.
Roughly speaking, the idea is to take a \Sze{} partition of
the graph{} and use the property to show that the
\Sze{} partition has almost constant densities.

The main purpose of this paper is to point out that these, and other
similar, characterizations of \qr{} graphs 
alternatively can be proved by 
replacing the \SRL{} and \Sze{} partitions by 
graph limit theory.
The idea is to
first take a graph limit of the sequence (or, in general, of a
subsequence) and a representing graphon, then
the property we assume of the graphs is translated into a property of
the graphon, and finally it is proved that this graphon
then has to be (\aex) constant. 
We do this for several different related characterizations below.
Our proofs will all have the same structure and consist of three
parts,
considering a sequence of graphs $(G_n)$ and a graphon $W$
with $G_n\to W$:
\begin{romenumerate}
  \item
An equivalence between a
condition on subgraph counts in $G_n$ and a corresponding
condition for integrals of a functional $\Psi$ of $W$.
($\Psi$ is a function on $\oi^m$ for some $m$,
and is a polynomial in $W(x_i,x_j)$, $1\le i<j\le m$.)
  \item
An equivalence between this integral condition on $\Psi$ and a pointwise
condition on  $\Psi$.
  \item
An equivalence between this pointwise
condition on $\Psi$ and $W=p$.
\end{romenumerate}
In all cases that we consider, (i) is rather straightforward, and performed
in essentially the same way for all versions. Step (ii) follows from
some version of the Lebesgue differentiation theorem, although some
cases are more complicated than others.
The arguments used in  (iii)
are similar to the arguments in earlier proofs that the
\Sze{} partition has almost constant densities
(under the corresponding condition on the graphs)
and the
algebraic problems that arise in some cases will be  the
same. However, the use of graph limits eliminates the many error terms
and $\eps$ inherent in arguments using the \SRL, and
provides at least sometimes proofs that are simpler and cleaner.
With some simplification, we can say that 
we split the proofs into
three parts (i)--(iii) which are combinatorial, analytic and
algebraic, respectively. This has the advantage of isolating 
different types of technical difficulties; moreover, it allows us to reuse
some steps that are the same for several different cases.
(See for example \refS{Svar} where we prove several variants
of the characterizations by modifying step (i) or (ii).)
On the other hand, it has to be admitted that there can be technical
problems with the analysis of the graphons too, especially in (ii), and that
our approach does not simplify the algebraic problems in (iii).
(In particular, we have not been able to improve the results in
\cite{SS:ind}, where it is this algebraic part that has not
yet been done for general graphs.)
Somewhat disappointingly, it seems that the graph limit method offers
greatest simplifications in the simplest cases.
At the end, it is partly a matter of taste if one prefers the
finite arguments using the \SRL{} or the infinitesimal
arguments using graphons; we invite the reader to make comparisons.

\begin{ack}
This work was begun during the workshop on 
Graph Limits, Homomorphisms and Structures in
Hrani\v cn\'i Z\'ame\v cek, Czech Republic, 2009.
We thank  
Asaf Shapira,
Miki Simonovits, Vera S\'os
and Bal\'azs Szegedy
for interesting discussions.
\end{ack}

\section{Preliminaries and notation}\label{Snot}

All graphs in this paper are finite, undirected and simple.
The vertex and edge sets of a graph $G$ are denoted by
$V(G)$ and $E(G)$. We write $|G|\=|V(G)|$ for the
number of vertices of $G$, and $e(G)\=|E(G)|$ for the
number of edges.
$\cp G$ is the complement of $G$.
As usual, $\nn\=\set{1,\dots,n}$.

\subsection{Subgraph counts}\label{SSsub}
Let $F$ and $G$ be graphs. It is convenient to assume that
the graphs are labelled, with
$V(F)=[\ff]\=\set{1,\dots,\ff}$, but the labelling does
not affect our results.
We define $N(F,G)$ as the number of labelled (not necessarily
induced) copies of $F$ in $G$; equivalently, $N(F,G)$ is the
number of injective maps $\gf:V(F)\to V(G)$ that are graph
homomorphisms
(\ie, if $i$ and $j$ are adjacent in $F$, then
$\gf(i)$ and $\gf(j)$ are adjacent in $G$).
If $U$ is a subset of $V(G)$, we further define
$N(F,G;U)$ as the number of such copies with all vertices in
$U$; thus $N(F,G;U)=N(F,G|_{U})$.
More generally, if $\uuf$ are subsets of $V(G)$, we
define
$N(F,G;\uuf)$ to be the number of 
labelled copies of $F$ in $G$ with the $i$th vertex
in $U_i$; equivalently, $N(F,G;\uuf)$ is the number of
injective graph homomorphisms 
$\gf:F\to G$ such that $\gf(i)\in U_i$ for every 
$i\in V(F)$.

\subsection{Quasi-random graphs}\label{SSqrg}
One of the several equivalent definitions of \qr{} graphs by
\citet{ChungGW:quasi} is:
\gnq{} (with $|G_n|\to\infty$)
is \pqr{} if and only if, 
for every graph $F$,
\begin{equation}
  \label{pqr}
N(F,G_n)= (p\eff+o(1)) |G_n|^\ff.
\end{equation}
(All unspecified limits in this paper are as \ntoo, and $o(1)$
denotes a quantity that tends to $0$ as \ntoo.
We will often use $o(1)$
for quantities that depend on some subset(s)
of a vertex set $V(G)$ or of $\oi$; we then always
implicitly assume that the convergence is uniform for all choices of
the subsets. 
We interpret $o(a_n)$ for a given sequence $a_n$ similarly.)

It turns out that it is not necessary to require \eqref{pqr}
for all graphs $F$; in particular, it suffices to use the graphs $K_2$ and
$C_4$ \cite{ChungGW:quasi}. However, it is not enough to require \eqref{pqr}
for just one graph $F$. As a substitute, \citet{SS:nni} showed that a
hereditary version of \eqref{pqr} for a single 
$F$ is sufficient; see \refS{Ssub}.

\subsection{Graph limits}\label{SSlimits}
The graph limit theory is also based on the subgraph counts
$N(F,G)$ (or the asymptotically equivalent number counting not
necessarily injective graph homomorphisms $F\to G$, see \cite{LSz,BCLSV1}).
A sequence \gnq{} of graphs, with
$|G_n|\to\infty$, \emph{converges},
if the numbers $\tinj(F,G_n)\=N(F,G_n)/(|G_n|)_\ff$
converge as \ntoo, for every fixed graph $F$.
(Here, $(|G_n|)_\ff$ denotes the falling factorial, which is the total number
of injective maps $V(F)\to V(G_n)$, so $\tinj(F,G_n)$ is
the proportion of injective maps that are homomorphisms. Since we
consider limits as $|G_n|\to\infty$ only, 
we could as well instead consider 
$t(F,G_n)$, 
the proportion of \emph{all} maps $V(F)\to V(G_n)$
that are homomorphisms, or the hybrid version
$N(F,G_n)/|G_n|^\ff$.)
Note that the numbers $\tinj(F,G_n)\in[0,1]$, which
implies the compactness property that every sequence \gnq{} of graphs
with $|G_n|\to\infty$
has a convergent subsequence.
For details and several other equivalent properties, see
\citet{LSz} and
\citet{BCLSV1,BCLSV2}; see also \citet{SJ209}.

The graph limits that arise in this way may be thought of as
elements of a completion of the space of (unlabelled) graphs with a
suitable metric. One useful 
representation \cite{LSz,BCLSV1} uses 
a symmetric measurable function $W:\oi^2\to\oi$;
such a function is called a \emph{graphon}, and defines a graph limit
in the following way.
If $F$ is a graph and $W$ a graphon, we define
\begin{equation}\label{psifw}
  \psifw(\xxf)\=\prod_{ij\in E(F)} W(x_i,x_j)
\end{equation}
and
\begin{equation}\label{tfw}
  t(F,W)\=\int_{\oiff}\psifw.
\end{equation}
(All integrals in this paper are with respect to the Lebesgue
measure in one or several dimensions, unless, in the appendix, we
specify another measure.)
A sequence \gnq{} converges to
the graph limit defined by $W$ if
$|G_n|\to\infty$ and 
\begin{equation}\label{tow}
\lim_\ntoo\tinj(F,G_n)= t(F,W)  
\end{equation}
(or, equivalently, $t(F,G_n)\to t(F,W)$)  
for every
$F$; as said above, in this case we write $G_n\to W$, although it
should be remembered that the
representation of the limit by a graphon $W$ is not unique.
(See \cite{BCLSV1,BCL:unique,SJ209,BR} for details on the
non-uniqueness. Note that, trivially, we may change $W$
on a null set without affecting the corresponding graph limit;
moreover, we may, for example, rearrange $W$ as in \eqref{arr} below.)

For example, the condition \eqref{pqr} can be written
$\tinj(F,G_n)\to p\eff$. Since the constant graphon
$W=p$ has $t(F,W)=p\eff$ for every $F$ by
\eqref{psifw}--\eqref{tfw}, this shows that, as said
in \refS{Sintro}, $\gnq$ is \pqr{} if and
only if $G_n\to p$.

\subsection{Graphons from graphs}\label{SSgg}
If $G$ is a graph, we define a corresponding graphon $W_G$ by
partitioning $\oi$ into $|G|$ intervals $I_i$ of
equal lengths $1/|G|$;
we then define $W_G$ to be 1 on every
$I_i\times I_j$ such that $ij\in E(G)$, and 0
otherwise.
It is easily seen that if $G$ is a graph, then 
\begin{equation}
  N(F,G)=|G|^\ff\int_{\oiff}\psifx{W_G}+O\bigpar{|G|^{\ff-1}}.
\end{equation}
(The error term is because we have chosen to count injective
homomorphisms only, \cf{} \cite{LSz,BCLSV1}.)
More generally, 
if $\uuf$ are subsets of $V(G)$ and $\uuif$ are
the corresponding subsets of $\oi$ given by
$U'_i\=\bigcup_{j\in U_i}I_j$, then
\begin{equation}\label{nfgu}
  N(F,G;\uuf)=|G|^\ff\int_{\uuifx}\psifx{W_G}+O\bigpar{|G|^{\ff-1}}.
\end{equation}

\subsection{Induced subgraph counts}\label{SSind}
In analogy with \refSS{SSsub} we define, for labelled graphs $F$ and $G$,
 $\Nq(F,G)$ as the number of \emph{induced} labelled 
copies of $F$ in $G$; equivalently, $\Nq(F,G)$ is the
number of injective maps $\gf:V(F)\to V(G)$ such
that $i$ and $j$ are adjacent in $F$ $\iff$
$\gf(i)$ and $\gf(j)$ are adjacent in $G$.
We further define
$\Nq(F,G;U)$ as the number of such copies with all vertices in
$U$ and
$N(F,G;\uuf)$ as the number of 
induced labelled copies of $F$ in $G$ with the $i$th vertex
in $U_i$. (Here $U,\uuf\subseteq V(G)$.)

For a graphon $W$ we make the corresponding definitions,
\cf{} \refSS{SSlimits},
\begin{equation}\label{psiqfw}
  \psiqfw(\xxf)\=\prod_{ij\in E(F)} W(x_i,x_j) 
\prod_{ij\not\in E(F)}\bigpar{1- W(x_i,x_j)}
\end{equation}
and
\begin{equation}\label{tindfw}
  \tind(F,W)\=\int_{\oiff}\psiqfw.
\end{equation}
Then, for any graph $G$, in analogy with \eqref{nfgu}
and using the notation there,
\begin{equation}\label{nfguq}
  \Nq(F,G;\uuf)=|G|^\ff\int_{\uuifx}\psiqfx{W_G}+O\bigpar{|G|^{\ff-1}}.
\end{equation}

\begin{remark}
If we define $\tind(F,G)\=\Nq(F,G)/(|G|)_\ff$, then
the convergence criterion \eqref{tow} (for every $F$) is equivalent to  
$\tind(F,G_n)\to\tind(F,W)$ (for every $F$) by
inclusion-exclusion \cite{LSz,BCLSV1}. 
\end{remark}

\subsection{Cut norm and cut metric}\label{SScut}

The \emph{cut norm} 
$\cn{W}$ of  $W\in L^1(\oi^2)$  is defined by
\begin{equation}\label{cut}
 \cn W \= \sup_{S,T\subseteq\oi}
  \lrabs{\int_{S\times T} W(x,y)\dd x\dd y}.
\end{equation}
A \emph{rearrangement} of the graphon $W$ is any graphon $W^\gf$ defined by
\begin{equation}\label{arr} 
 W^\gf(x,y) = W(\gf(x),\gf(y)),
\end{equation}
where $\gf:\oi\to\oi$ is a measure-preserving bijection.
The \emph{cut metric} $\gd$
by \citet{BCLSV1}
may be defined by, for two graphons $W_1,W_2$,
\begin{equation}\label{cutdef}
 \dcut(W_1,W_2) = \inf_\gf \cn{W_1-W_2^\gf},
\end{equation}
where the infimum is over all rearrangements of $W_2$.
(It makes no
difference if we rearrange $W_1$ instead, or both $W_1$ and $W_2$.)

A major result of \citet{BCLSV1} is that if
$|G_n|\to\infty$, then
$G_n\to W \iff\dcut(W_{G_n},W)\to0$, so convergence of a sequence of
graphs
as defined above is the same as convergence in the metric $\dcut$.

\section{Subgraph counts in induced subgraphs}\label{Ssub}

\citet{SS:nni} gave the following characterization of
\pqr{} graphs using the numbers of subgraphs of a given type
in induced subgraphs.
(The case $F=K_2$, when $N(K_2,G_n;U)$ is twice the number of edges
with both endpoints in $U$, is one of the original \qr{} properties in
\cite{ChungGW:quasi}.) 

\begin{theorem}[\citet{SS:nni}]
  \label{T2}
Suppose that $(G_n)$ is a sequence of graphs with $|G_n|\to\infty$.
Let $F$ be any fixed graph with $e(F)>0$ and let $0<p\le1$.
Then $(G_n)$ is \pqr{} if and only if,
for all subsets $U$ of\/ $V(G_n)$,
\begin{equation}\label{t2}
  N(F,G_n;U)=p^{e(F)}|U|^\ff+o\bigpar{|G_n|^\ff}.
\end{equation}
\end{theorem}

For our discussion of graph limit method, it is also interesting to
consider the 
following weaker version (with a stronger hypothesis), patterned after
\refT{T3} below. 

\begin{theorem}
  \label{T1}
Suppose that $(G_n)$ is a sequence of graphs with $|G_n|\to\infty$.
Let $F$ be any fixed graph with $e(F)>0$ and let $0<p\le1$.
Then $(G_n)$ is \pqr{} if and only if,
for all subsets $\uuf$ of\/ $V(G_n)$,
\begin{equation}\label{t1}
  N(F,G_n;\uuf)=p^{e(F)}\prodif|U_i|+o\bigpar{|G_n|^\ff}.
\end{equation}
\end{theorem}

\begin{remark}\label{RT1a}
  Since \eqref{t2} is the special case of \eqref{t1}
  with $U_1=\dots=U_\ff$, the 'if' direction of
  \refT{T1} is a corollary of \refT{T2}. The 'only
  if' direction does not follow immediately from \refT{T2},
  but it is straightforward to prove, either by the methods of
  \cite{SS:nni} or by our methods with graph limits, see \refS{SpfT1};
  hence the main interest is in the 'if' direction. (The same is true
  for the results below for the induced case.)
\end{remark}

\begin{remark}\label{RT1}
Theorems \refand{T2}{T1} obviously fail when $e(F)=0$, since then
\eqref{t2} and \eqref{t1} hold trivially and
  the assumptions give no information on $G_n$. They fail also if
  $p=0$;
for example, if $F=K_3$ and $G_n$ is the complete
  bipartite graph $K_{n,n}$.
\end{remark}

\citet{Shapira} and \citet{ShapiraY} consider also an
intermediate version where a 
symmetric form of \eqref{t1} is used, summing over all
permutations of $(\uuf)$ (or, equivalently, over all
labellings of $F$); moreover, $\uuf$ are supposed to be
disjoint and of the same size. It is shown directly 
in \cite{Shapira} that this is equivalent to \eqref{t2}.
See also Subsections \refand{SSdisjoint}{SSsamesize}.

The main result of \citet{Shapira} is that \refT{T2} remains valid even
if we only require \eqref{t2} for $U$ of size $\gc\gnabs$ with $\gc=1/(\ff+1)$.
(It is a simple consequence that any smaller positive $\gc$ will also do.)
This was improved by \citet{Yuster}, who proved this for any $\gc\in(0,1)$.
We state this, 
and the corresponding result for a sequence of (disjoint) subsets.

\begin{theorem}[\citet{Yuster}]
  \label{T2c}
Let $\gnq$, $F$ and $p$ be as in \refT{T2}, and let $0<\gc<1$. Then
$\gnq$ is \pqr{} if and only if \eqref{t2} holds for all subsets $U$
of $V(G_n)$ with $|U|=\floor{\gc\gnabs}$.
\end{theorem}

\begin{theorem}
  \label{T1c}
Let $\gnq$, $F$ and $p$ be as in \refT{T1}, and let $0<\gc<1$. Then
$\gnq$ is \pqr{} if and only if \eqref{t1} holds for all subsets $\uuf$
of $V(G_n)$ with $|U_i|=\floor{\gc\gnabs}$.

If $\gc<1/\ff$, it is enough to assume \eqref{t1} for $\uuf$ that
further are disjoint.
\end{theorem}

For $F=K_2$, \refT{T2c} with $\gc=1/2$ is another of the original
characterizations by \citet{ChungGW:quasi}, and the generalization to
arbitrary $\gc\in(0,1)$ is stated in \citet{ChungG}. Another related
characterization from \cite{ChungG} is discussed in
\refS{Scut}.

Turning to induced copies of $F$, the situation
is much more complicated, as discussed in \citet{SS:ind}.
First, the expected number of induced labelled copies of $F$ in a random
graph $G(n,p)$ is $\gb_F(p)n^\ff+o(n^\ff)$, with
\begin{equation}\label{beta}
  \gb_F(p)\=p\eff(1-p)^{e(\cp F)}
=p\eff(1-p)^{\binom{\ff}2-e(F)}.
\end{equation}
Hence, the condition corresponding to \eqref{t2} for induced
subgraphs is: 
For all subsets $U$ of $V(G_n)$,
\begin{equation}\label{tq}
  \Nq(F,G_n;U)=\gb_F(p)|U|^\ff+o\bigpar{|G_n|^\ff}.
\end{equation}
Indeed, as observed in \cite{SS:nni,SS:ind}, this holds for
every \pqr{} $(G_n)$, but the converse is generally
false.
One reason is that, provided $F$ is neither empty nor complete, then
$\gb_F(0)=\gb_F(1)=0$, and if
$p_F\=e(F)/\binom{\ff}2$ (the edge density in
$F$), then $\gb_F(p)$ increases on $[0,p_F]$ and
decreases on $[p_F,1]$. Hence, for every $p\neq p_F$,
there is another $\cpp$ such that $\gb_F(\cp
p)=\gb_F(p)$; we call $p$ and $\cpp$
\emph{conjugate}.
(For completeness, we let $\cpp\=p$ when $p=p_F$ or
when $F$ is empty or complete. Note also that $\cpp$
depends on $F$ as well as $p$.)
Obviously, a \cpqr{} sequence $(G_n)$ also satisfies
\eqref{tq}. Moreover, any combination of a \pqr{}
sequence and a \cpqr{} sequence will satisfy
\eqref{tq}. Hence the best we can hope for is the following.
We say that $(G_n)$ is \emph{\mppqr} if it is \pqr, \cpqr,
or a combination of two such sequences.

\begin{definition}\label{DHF}
Let $0\le p\le 1$.
We say that a graph $F$ is \emph{hereditary 
induced-forcing} ($\HFp$)
if every
$(G_n)$ that satisfies \eqref{tq} for all subsets
$U$ of $V(G_n)$ 
is \mppqr.
In this case we also write $F\in\HFp$ (thus regarding
$\HFp$ as a set of graphs).

We say that $F$ is $\HF$ (and write
$F\in\HF$) if $F$ is $\HFp$ for every
$p\in(0,1)$ (thus excluding the rather exceptional cases
$p=0$ and $p=1$).
\end{definition}

\begin{remark}\label{Rmixed}
  The definition of \mppqr{} is perhaps better stated in
  terms of graph limits.
Just as $(G_n)$ is \pqr{} if and only if $G_n\to
  p$, where $p$ stands for the graphon that is constant
  $p$,
$(G_n)$ is \mppqr{} if and only if the limit points of
  $(G_n)$ are contained in \set{p,\cpp}, \ie, if
  every convergent subsequence of $(G_n)$ converges to either
the graphon $p$ or the graphon $\cpp$.

In general we say that a sequence $(G_n)$, with $|G_n|\to\infty$ as always,
is \emph{mixed \qr} if the set of 
limit points is contained in \set{p:p\in\oi}, \ie,
if every convergent subsequence converges to a constant
graphon. (Equivalently, if every convergent subsequence is \qr).
\end{remark}

\begin{remark}
  \label{Rqrw}
Just as one talks about \qr{} properties of graphs, or more
properly of sequences $(G_n)$ of graphs, we say that a
property of graphons $W$ is \emph{\pqr} if it is
satisfied only by $W=p$ \aex, 
that it is \emph{\qr} if it is \pqr{} for some
$p\in\oi$, and that it is \emph{mixed
\qr} if it is satisfied only by graphons that are
\aex{} constant (for some set of accepted constants).
\end{remark}

\citet{SS:ind} gave a counter-example showing that the path
$P_3$ with 3 vertices is \emph{not} $\HF$. They also
showed that every regular $F$ (with $\ff\ge2$) is $\HF$, 
and conjectured that $P_3$ and its
complement $\cp P_3$ are the only graphs not in $\HF$.
This conjecture remains open. (The methods of the present paper do
not seem to help.)

\begin{remark}
  The cases $F$ empty or complete are exceptional and rather trivial. If $F$
  is complete graph $K_m$ ($m\ge2$), then
$\Nq(F,G_n;U)=N(F,G_n;U)$, and 
thus \eqref{tq} implies that $(G_n)$ is \pqr{}
by \refT{T2} (but not for $p=0$
  unless $m=2$, see \refR{RT1}).
By taking complements we see that the same holds for
for an empty graph $E_m$ ($m\ge2$) and $0\le p<1$.

In particular, $E_m,K_m\in\HF$ when $m\ge2$.
\end{remark}

In view of the fact that not all graphs are $\HF$, 
\citet{ShapiraY} gave the following substitute, which is an induced
version of \refT{T1}.

\begin{theorem}[\citet{ShapiraY}]
  \label{T3}
Suppose that $(G_n)$ is a sequence of graphs with $|G_n|\to\infty$.
Let $F$ be any fixed graph with $|F|>1$ and let $0<p<1$.
Then $(G_n)$ is \mppqr{} if and only if,
for all subsets $\uuf$ of\/ $V(G_n)$,
\begin{equation}\label{t3}
  \Nq(F,G_n;\uuf)
=p^{e(F)}(1-p)^{\binom{\ff}2-e(F)}\prodif|U_i|+o\bigpar{|G_n|^\ff}.
\end{equation}
Moreover, it suffices that \eqref{t3} holds for all sequences
$\uuf$ of  disjoint subsets of\/ $V(G_n)$ with
the same size, $|U_1|=\dots=|U_\ff|$.
\end{theorem}

To show the flexibility with which our method combines different
conditions, we also show
that it suffices to consider subsets
of a given size for induced subgraph counts too,
in analogy with Theorems \refand{T2c}{T1c}.

\begin{theorem}
  \label{T3x}
In \refT{T3}, 
it suffices that \eqref{t3} holds for all sequences
$\uuf$ of subsets of\/ $V(G_n)$ with
$|U_i|=\floor{\gc\gnabs}$, for any fixed $\gc$ with
$0<\gc<1$.
Alternatively, if $0<\ga<1/\ff$, it suffices that
\eqref{t3} holds for all such sequences of  disjoint $\uuf$.
\end{theorem}

\begin{theorem}
  \label{T3y}
Let $0<\gc<1$ and $0\le p\le 1$, and
let $F$ be a fixed graph with $F\in\HFp$. 
Then every sequence $(G_n)$ with
$|G_n|\to\infty$ such that 
\eqref{tq} holds for all 
subsets $U$ of\/ $V(G_n)$ with
$|U|=\floor{\gc\gnabs}$ is \mppqr.
\end{theorem}

\begin{remark}
Theorems \refand{T1c}{T3x} fail for disjoint sets
$\uuf$ in the limiting case $\gc=1/\ff$, at least
for $F=K_2$, see \refS{Scut} and \refR{Rl3}.
We leave it as an open problem to investigate this case for other
graphs $F$.
\end{remark}
 
\section{Graph limit proof of \refT{T1}}\label{SpfT1}

We give proofs of the theorems above using graph limits; the reader
should compare these to the combinatorial proofs in
\cite{SS:nni,SS:ind,Shapira,ShapiraY,Yuster} using the \SRL.
In order to exhibit the main ideas clearly, we begin in this section
with the simplest case and
give a detailed proof of \refT{T1}. In the
following sections we will give the minor modifications needed for the
other results, treating the additional complications one by one.

The first step is to recall that the space of graphs and graph limits
is compact; thus, every sequence has a convergent subsequence
\cite{BCLSV1}. Hence, if $(G_n)$ is not \pqr, we can
select a subsequence (which we also denote by $(G_n)$), such
that $G_n\to W$ for some graphon $W$ that is not
equivalent to the constant graphon $p$, which simply means that
$W\neq p$ on a set of positive measure.

Hence, in order to prove \refT{T1}, it suffices to assume that further
$G_n\to W$ for some graphon $W$, and then prove that $W=p$ a.e.

\subsection{Translating to graphons}\label{SStrans}

In this subsection we use the graph limit theory in \cite{BCLSV1}
to translate the property \eqref{t1} to graph limits.

We begin with an easy consequences of 
Lebesgue's differentiation theorem; for future reference we state it
as a (well-known) lemma. (See \refL{L3} below for a stronger version.)
We let $\gl$ denote Lebesgue measure (in one or several dimensions).

\begin{lemma}
  \label{L1}
Suppose that $f:\oi^m\to\bbR$ is an integrable function
such that $\intaamx f=0$ for all sequences $\aam$ of
disjoint measurable subsets of\/ $\oi$. Then $f=0$ \aex

Moreover, it is enough to consider $\aam$ with
$\gl(A_1)=\dots=\gl(A_m)$; 
we may even further impose that $\gl(A_k)\in \set{\eps_1,\eps_2,\dots}$ for
any given sequence $\eps_n\to0$.
\end{lemma}

\begin{proof}
  For any distinct $\xxm\in(0,1)$ and any sufficiently
  small $\eps>0$ we take
  $A_i=(x_i-\eps,x_i+\eps)$ and find
  \begin{equation*}
(2\eps)^ {-m} \int_{|y_i-x_i|<\eps,\,i=1,\dots,m}  f(\yym) 
= (2\eps)^ {-m} \int_{\aamx}  f=0.
  \end{equation*}
By Lebesgue's differentiation theorem,
see \eg{} \citet[\S1.8]{Stein},
the \lhs{} converges to $f(\xxm)$ as
$\eps\to0$ for \aex{} $\xxm$.
\end{proof}

We can now easily translate the condition \eqref{t1} in \refT{T1} to a
corresponding condition for the limiting graphon (which we may assume
exists, as discussed above).

\begin{lemma}
  \label{LA1}
Suppose that $G_n\to W$ for some graphon $W$ and let
$F$ be a fixed graph and $\gax\ge0$ a fixed number.
Then the following are equivalent:
\begin{romenumerate}
  \item
For all subsets $\uuf$ of\/ $V(G_n)$, 
\begin{equation}\label{la1i}
  N(F,G_n;\uuf)=\gax\prodif|U_i|+o\bigpar{|G_n|^\ff}.
\end{equation}
\item
For all subsets $\aaf$ of\/ $\oi$,
\begin{equation}\label{la1ii}
  \intaafx \psifw(\xxf)=\gax\prodif\gl(A_i).
\end{equation}
\item
$\psifw(\xxf)=\gax$ for \aex{} $\xxf\in\oiff$.
\end{romenumerate}
\end{lemma}

\begin{proof}
(iii)$\implies$(ii) is trivial, and
(ii)$\implies$(iii) is immediate by \refL{L1}
  applied to $\psifw-\gax$.

(i)$\iff$(ii).
  The convergence $G_n\to W$ is equivalent to
  $\dcut(W_{G_n},W)\to0$.
By the definition of $\dcut$, there thus exist
  \mpx{} bijections $\gf_n:\oi\to\oi$ such
  that if  $W_n\=W_{G_n}^{\gf_n}$, then
$\cn{W_n-W}\to0$. 
Fix $n$, and
let $\inj$ ($1\le j\le n)$ be the intervals of length $|G_n|\qw$ used to
  define $W_{G_n}$, and let as in \eqref{nfgu}
  $U'\=\bigcup_{j\in U} \inj$ for a subset
  $U$ of $V(G_n)$; 
further, let $\inj''\=\gf_n\qw(\inj)$ and 
$U''\=\gf_n\qw(U')=\bigcup_{j\in U} \inj''$.
Then, for any subsets $\uuf$ of $V(G_n)$,
by  \eqref{nfgu} and a change of variables,
\begin{equation*}
  N(F,G_n;\uuf)=|G_n|^\ff\int_{\uuiifx}\psifx{W_n}+o\bigpar{|G_n|^\ff}.
\end{equation*}
Hence, (i) 
is equivalent to
\begin{equation}\label{e1'}
  \int_{\uuiifx}\psifx{W_n}
=
\gax\prodif\frac{|U_i|}{|G_n|}+o(1)
=
\gax\prodif\gl(U''_i)+o(1),
\end{equation}
for all subsets $U_i''$ that are unions of sets $\inj''$.

We next extend \eqref{e1'} from the special sets $U''_i$
(in a family that depends on $n$) to arbitrary (measurable) sets.
Thus, assume that \eqref{e1'} holds, and let $\aaf$
be arbitrary subsets of \oi. 
Fix $n$ and 
let $a_{ij}\=\gl(A_i\cap
\injy)/\gl(\injy)$.
Further, let $B_i$ be a random subset of \oi{} obtained
by taking an independent family $J_{ij}$ of independent
0--1 random variables with
$\P(J_{ij}=1)=a_{ij}$, and then taking $B_i\=\bigcup_{j:J_{ij}=1}\injy$.
Then the sets $B_i$ are of the form $U''_i$, so
\eqref{e1'} applies to them, and,
noting that $W_n$ is constant on every set $\iniy\times \injy$, 
and hence $\psifwn$ is constant on every set $\ijifx$,
\begin{equation}\label{sofie}
  \begin{split}
\int_{\aafx} (\psifwn-\gax)
&=
\sum_{j_1,\dots,j_\ff=1}^{|G_n|}\prodif a_{ij_i}\int_{\ijifx}(\psifwn-\gax)
\\&
=\E\sum_{j_1,\dots,j_\ff=1}^{|G_n|}\prodif J_{ij_i}\int_{\ijifx}(\psifwn-\gax)
\\&
=
\E \int_{\bbfx} (\psifwn-\gax)
=o(1),
  \end{split}
\raisetag{1.4\baselineskip}
\end{equation}
where the final estimate uses \eqref{e1'}.
Consequently, \eqref{e1'}, for all special sets
$U_i''$, is equivalent to the same estimate 
\begin{equation}\label{e1''}
  \int_{\aafx}\psifx{W_n}
=
\gax\prodif\gl(A_i)+o(1),
\end{equation}
for any measurable sets $\aaf$ in
$\oi$. Consequently, (i) is equivalent to
\eqref{e1''}.
(Recall that estimates such as \eqref{e1''} are supposed to
be uniform over all choices of $\aaf$.)

It is well-known that for two graphons $W$ and $W'$,
\begin{equation*}
\biggabs{\int_{\oi^m}\bigpar{\psifx{W}-
\psifx{W'}}}=O(\cn{W-W'}),  
\end{equation*}
see \cite{BCLSV1};
moreover, the  proof in \cite{BCLSV1}
(or the version of the proof in \cite{BR})
shows that the same holds, uniformly, also if we
integrate over a subset $\aamx$.
(In other words, extending the cut norm to functions of several variables
as in \cite{cutsub}, 
$\cn{\psifx{W}- \psifx{W'}}=O(\cn{W-W'})$.)
Consequently, 
the assumption $G_n\to W$, which as said yields $\cn{W_n-W}\to0$, implies that 
$\int_{\aafx}\psifx{W_n}=\int_{\aafx}\psifx{W}+o(1)$, and thus
\eqref{e1''}, and hence (i),
is equivalent to
\begin{equation}\label{e1}
  \int_{\aafx}\psifx{W}
=
\gax\prodif\gl(A_i)+o(1).
\end{equation}
Consequently, (ii)$\implies$(i). 
Conversely, none of the terms in \eqref{e1} depends on
$n$, so if \eqref{e1} holds, then the $o(1)$
error term vanishes and \eqref{la1ii} holds. Hence
(i)$\implies$(ii). 
\end{proof}

\subsection{An optional measure theoretic interlude}\label{SSnull}

To prove \refT{T1}, it thus remains only to show that if
$W$ is a graphon such that $\psifw=p\eff$ \aex,
then $W=p$ a.e.
(In the terminology of \refR{Rqrw},
``$\psifw=p^{e(F)}$'' is a \pqr{} property.)

We know several ways to do this. One, direct, is given in \refSS{SSC1ae}.
However, as will be seen  in \refSS{SSC1},
it is much simpler to argue if we can assume that
$\psifw=p\eff$ everywhere, and not just a.e. (The main
reason is that we then can choose $x_1=x_2=\dots=x_\ff$.)
Hence, somewhat
surprisingly, 
the qualification 'a.e.' here forms a significant technical
problem. Usually, '\aex' is just a technical formality in
arguments in integration and measure theory, but here it is an
obstacle and we would like to get rid of it. 
We do not see any trivial way to do
this, but we can do it as follows. 
(To say that $W'$ is a version of $W$ means that
$W'=W$ \aex; this implies that all integrals considered here
are equal for $W$ and $W'$, and thus $G_n\to W'$
as well.)
See \refSS{SS1A} and \refApp{Appa} for an alternative.

\begin{lemma}
  \label{LB1} Let $F$ be a graph with $e(F)>0$, and let
  $W$ be a graphon.
If $\psifw=\gax>0$ \aex{} on $\oiff$,
then there exists a version $W'$ of\/ $W$ such that
$\psifx{W'}(\xxf)=\gax$ for all $(\xxf)\in\oiff$.
\end{lemma}

\begin{proof}
By  symmetry, we may assume $12\in E(F)$; hence
  $\psifw(\xxf)$, defined in \eqref{psifw}, contains a
  factor $W(x_1,x_2)$. We 
let $x'\=(x_3,\dots,x_\ff)$ and 
collect the other factors in
  \eqref{psifw} into a product $f(x_1,x')$ of the
  factors corresponding to edges $1j\in E(F)$ with
  $j\ge3$, and another product $g(x_2,x')$ of the
  remaining factors.
Thus
\begin{equation*}
  \psifw(\xxf)=W(x_1,x_2)f(x_1,x')g(x_2,x').
\end{equation*}
By assumption, thus
\begin{equation}\label{w1}
W(x_1,x_2)f(x_1,x')g(x_2,x')=\gax
\end{equation}
for \aex{} $(x_1,x_2,x')$. We may thus choose
$x'$ (\aex{} choice will do) such that
\eqref{w1} holds for \aex{} $(x_1,x_2)$.
We fix one such $x'$ and write $f(x)\=f(x,x')$, $g(y)\=g(y,x')$;
we then have $W(x,y)f(x)g(y)=\gax$ for \aex{}
$(x,y)$.

We define
$W_1(x,y)\=\max\bigpar{1,\gax/(f(x)g(y))}$;
thus $W_1=W$ a.e.

Let $|(\xxm)|_\infty\=\max|x_i|$.
Recall that if $f$ is an integrable function on
$\bbR^m$ for some $m$ (or on a subset such as
$\oi^m$), then a point $x$ is a \emph{\lp}
of $f$ if
$(2\eps)^{-m}\int_{|y-x|_\infty<\eps}|f(y)-f(x)|\dd
y =o(1)$ as $\eps\to0$. In probabilistic terms, this says
that if $\xex$ is a random point in the cube 
\set{y:|y-x|_\infty<\eps}, then
$f(\xex)\lto f(x)$.
For bounded functions, which is the case here, this is equivalent to
$f(\xex)\pto f(x)$ as $\eps\to0$, which shows,
for example, that if $x$ is a \lp{} of both $f$
and $g$, then it is also a \lp{} of $f\pm g$,
$fg$, and, provided $g(x)\neq0$, of $f/g$.
It is well-known,
see \eg{} \citet[\S1.8]{Stein},
that if $f$ is integrable, then \aex{} point is a
\lp{} of $f$.

We can thus find a null set $N\subset\oi$ such that every
$x\in\sss\=\oi\setminus N$ is a \lp{} of
both $f$ and $g$.
Since $W(x,y)\le1$ and thus $f(x)g(y)\ge\gax$
\aex, 
it then follows that if $(x_1,x_2)\in\sss^2$, then
$(x_1,x_2)$ is a \lp{} of $W_1$. This implies, by the definition
\eqref{psifw}, that if $(\xxf)\in\sss^\ff$,
then $(\xxf)$ is a \lp{} of $\psifw$;
hence, using $\psifx{W_1}=\psifw$ \aex{}
and $\psifw=\gax$ \aex,
$\psifx{W_1}(\xxf)=\gax$ for $(\xxf)\in\sss^\ff$.

This would really be enough for our purposes, but to obtain the
conclusion as stated, we choose $x_0\in\sss$ and define
$\gf:\oi\to\oi$ by $\gf(x)=x$ for
$x\in\sss$ and $\gf(x)=x_0$ for $x\in
N$; then $W'\=W_1^\gf$ satisfies
$\psifx{W'}=\gax$ everywhere.
\end{proof}

\begin{remark}\label{RB1}
  Although we do not need it, we note that \refL{LB1} is valid for the
  trivial case $e(F)=0$ too, since then $\psifw=1$ for every $W$ and
  there is nothing to prove. 
We do not know whether \refL{LB1} is also valid for $\gax=0$; consider
  for example $F=K_3$.  (In this case it suffices to
  consider 0/1-valued $W$ and $W'$.)
\end{remark}

\subsection{The first algebraic argument}\label{SSC1}

The proof of \refT{T1} is now completed, by Lemmas
\refand{LA1}{LB1} and the remarks above, by the
following lemma:

\begin{lemma}
  \label{LC1}
Let $F$ be a graph with $e(F)>0$ and let $W$ be a
graphon. If $p>0$ and $\psifw(\xxf)=p\eff$ for every
$(\xxf)\in\oiff$, then $W=p$.
\end{lemma}

\begin{proof}
  First take $x_1=x_2=\dots=x_\ff=x$. Then
  $\psifw(\xxf)=W(x,x)\eff$, and thus $W(x,x)=p$,
  for every $x\in\oi$.
Next, we may assume by symmetry that the degree $d_1$ of vertex
  1 in $F$ is non-zero. 
Let $x,y\in\oi$ and 
take $x_1=x$ and $x_2=\dots=x_\ff=y$. Then
\begin{equation*}
  p\eff=\psifw(\xxf)=W(x,y)^{d_1}W(y,y)^{e(F)-d_1}
=W(x,y)^{d_1}p^{e(F)-d_1}.
\end{equation*}
Hence $W(x,y)=p$.
\end{proof}

This completes the first version of our graph limit proof of \refT{T1}.

\subsection{The second algebraic argument}\label{SSC1ae}

As said above, we can alternatively avoid \refL{LB1} and
instead use the following stronger version of \refL{LC1},
which together with \refL{LA1} yields another proof of \refT{T1}.

\begin{lemma}
  \label{LC1ae}
Let $F$ be a graph with $e(F)>0$ and let $W$ be a
graphon. If $\psifw(\xxf)=p\eff$ for \aex{}
$(\xxf)\in\oiff$, then $W=p$ a.e.
\end{lemma}

\begin{proof}
  We first symmetrize. If $\gs\in\fsf$, the symmetric
  group of all permutations of \set{1,\dots,\ff}, let
  $\gs(F)$ be the image of $F$, with edges
  $\gs(i)\gs(j)$ for $ij\in E(F)$, and consider 
  \begin{equation*}
\begin{split}
\prod_{\gs\in\fsf}
	\Psi_{\gs(F),W}(\xxf)
=\prod_{1\le i<j\le\ff}
W(x_i,x_j)^{e(F)k!/\binom k2},	
\end{split}
  \end{equation*}
where the  equality follows because, by symmetry, each
$ij$ is an edge in $\gs(F)$ for $e(F)k!/\binom k2$ permutations $\gs$. 
By the assumption, this equals $p^{e(F)k!}$ \aex, so
taking logarithms and dividing by $e(F)k!$ we obtain
\begin{equation*}
\binom k2 \qw\sum_{1\le i<j\le\ff} \log W(x_i,x_j)
=\log p,
\qquad \text{a.e.}
\end{equation*}
For \aex{} $(x_1,\dots,x_{\ff+2})$, this
holds for every subsequence of $\ff$ elements $x_i$; it
then follows by \refL{Lassym} below, with $d=2$, $h=\ff$ and
$a(\set{i,j})=\log W(x_i,x_j)-\log p$, that in this case $W(x_1,x_2)=p$. 
Hence $W(x_1,x_2)=p$ for \aex{} $(x_1,x_2)$.
\end{proof}

\begin{lemma}\label{Lassym}
  Suppose that $1\le d\le h$, and let
  $a(I)$ be an array defined for all $d$-subsets $I$ of
  $[h+d]$.
Suppose further that for every $h$-subset $J$ of\/ $[h+d]$,
\begin{equation}\label{lassym}
  \sum_{I\subseteq J} a(I)=0,
\end{equation}
summing over the $\binom hd$ subsets of size $d$.
Then $a(I)=0$ for every $I$.
\end{lemma}
\begin{proof}
  This is a form of a result by \citet{Gottlieb}.
(It is easily proved by fixing a $d$-subset $I_0$ and then summing
  \eqref{lassym}
for all $J$ with $|J\cap I_0|=k$, for
$k=0,\dots,d$; we omit the details.)
\end{proof}

\subsection{Further proofs}\label{SS1A}

Instead of \refL{LB1} we may use the weaker but more general
\refL{LD} in \refApp{Appa}; this lemma, with
$\Phi((w\ij)_{i<j})\=\prod_{ij\in E(F)}w\ij$, yields a
version of $W$ such that
$\psifw(\xxf)=p^{e(F)}$ at enough points so that the
proof of \refL{LC1} applies for \aex{} $(x,y)$.
(Although \refL{LD} does not guarantee
$\psifw=p^{e(F)}$ everywhere as \refL{LB1}
does.) This and \refL{LA1} yield another proof of \refT{T1}.

Alternatively, we may use \refT{TD} and argue as in the proof
of \refL{LC1}, with only notational changes, to show that
\refT{TD}(iii) does not hold for this $\Phi$, and
hence by \ref{tdw}$\iff$\ref{tduvs} in
\refT{TD}, $W$ is \aex{} constant and thus
$W=p$ \aex, yielding another proof of 
\refL{LC1ae}, and thus of \refT{T1}.

A modification of this argument is to use \refL{LC1} as
stated together with \refC{CD} to conclude that
\refL{LC1ae} holds.

Any of these proofs of \refT{T1} thus uses only the simple
algebraic argument in 
\refL{LC1} but combines it with results from \refApp{Appa}.
The latter results have rather long and technical proofs, which is the
reason why we have postponed them to an appendix.
If the objective is only to prove \refT{T1}, the direct proof
of \refL{LB1} is much simpler than using \refL{LD}
or one of its consequences \refT{TD} or
\refC{CD}. However, we have here started with the simplest
case, and for other cases it seems much more complicated to prove
analogues of \refL{LB1} or \refL{LC1ae}
directly. Hence, our main method in the sequel will be to use the
results of \refApp{Appa}, which once proven and available do
not have to be modified.

Nevertheless, we have chosen to present also the direct proofs in
Subsections \ref{SSnull}--\ref{SSC1} and
\refSS{SSC1ae} in order to show alternative ways that in the present
case are simpler. We furthermore want to inspire readers to investigate
whether there are similar direct proofs (that we have failed to find)
in some of the cases treated later too.

\section{One subset: proof of \refT{T2}}\label{SpfT2}


We next give a proof of \refT{T2} along the lines of \refS{SpfT1}.
We begin with a lemma giving an analogue of \refL{L1} for the case
$A_1=\dots=A_m$. 

If $f$ is a function on $\oi^m$ for some $m$, we
let $\tf$ denote its symmetrization defined by
\begin{equation}
  \label{symm}
\tf(\xxm)\=\frac1{m!}\sum_{\gs\in\fS_m} f\bigpar{\xxgsm},
\end{equation}
where $\fS_m$ is the symmetric group of all $m!$
permutations of \set{1,\dots,m}. Note that for any
integrable $f$ and any subset
$A$ of \oi,
\begin{equation}
  \label{symm1}
\int_{A^m} \tf= \int_{A^m} f.
\end{equation}

\begin{lemma}
  \label{L2}
Suppose that $f:\oi^m\to\bbR$ is an integrable function
such that $\int_{A^m} f=0$ for all  measurable subsets
$A$ of\/ $\oi$. Then $\tf=0$ \aex
\end{lemma}

\begin{proof}
  Let $\aam$ be disjoint subsets of $\oi$. For any
  sequence $\xim\in\setoi^m$, take
  $A\=\bigcup_{i:\xi_i=1} A_i$. Then $\etta_A=\sumim\xi_i\etta_{A_i}$
and
\begin{equation}
  \label{l2a}
0=\int_{A^m} f = \int_{\oi^m} f\etta_{A^m}
=\sum_{i_1,\dots,i_m=1}^m\xi_{i_1}\dotsm\xi_{i_m}
\int_{A_{i_1}\times\dots\times A_{i_m}} f.
\end{equation}
The monomials $\xi_{i_1}\dotsm\xi_{i_k}$
with $i_1<\dots<i_k$, $0\le k\le m$, form a basis
of the $2^m$-dimensional space of functions on
$\setoi^m$. Hence, collecting terms in \eqref{l2a},
the coefficient of each such monomial vanishes. 
In particular, for the coefficient of
$\xi_1\dotsm\xi_m$ we obtain a contribution only when
$\iim$ is
a permutation of $1,\dots,m$, and we obtain
\begin{equation*}
  0=\sum_{\gs\in\fS_m}\int_{\aagsmx} f
=m!\int_{\aamx}\tf.
\end{equation*}
The result follows by \refL{L1}, applied to $\tf$.
\end{proof}

We can now translate the property \eqref{t2} to graphons, \cf{} \refL{LA1}.

\begin{lemma}
  \label{LA2}
Suppose that $G_n\to W$ for some graphon $W$ and let
$F$ be a fixed graph and $\gax\ge0$ a fixed number.
Then the following are equivalent:
\begin{romenumerate}
  \item
For all subsets $U$ of\/ $V(G_n)$, 
\begin{equation*}
  N(F,G_n;U)=\gax|U|^\ff+o\bigpar{|G_n|^\ff}.
\end{equation*}
\item
For all subsets $A$ of\/ $\oi$,
\begin{equation*}
  \intaf \psifw(\xxf)=\gax\gl(A)^\ff.
\end{equation*}
\item
$\tpsifw(\xxf)=\gax$ for \aex{} $\xxf\in\oiff$.
\end{romenumerate}
\end{lemma}
\begin{proof}
  This is proved almost exactly as \refL{LA1}, with obvious
  notational changes and with \refL{L1} replaced by
  \refL{L2}, which together with \eqref{symm1} implies (ii)$\iff$(iii).
The main difference is that we now use a single random set 
$B\=\bigcup_{j:J_{j}=1}\injy$, where $\set{J_j}$ is a
  family of independent indicator variables.
Hence, the analogue of \eqref{sofie} is not exact; we have
\begin{equation}\label{magnus}
\E\prodif J_{j_i}
= \prodif a_{j_i}
\end{equation}
when $j_1,\dots,j_\ff$ are distinct, but in general not when two or
more are equal. However, there are only
$O(|G_n|^{\ff-1})$ choices of indices with at least two
coinciding, and each such choice introduces an error that is at most 
$\gl(\ijifx)=|G_n|^{-\ff}$. Hence, 
we now have
\begin{equation}\label{cecilia}
\int_{A^n} (\psifwn-\gax)
=
\E \int_{B^n} (\psifwn-\gax)
+o(1).
\end{equation}
The error $o(1)$ is unimportant, and, assuming (i), the conclusion of
\eqref{sofie} is valid in the form
$\int_{A^n} (\psifwn-\gax)=o(1)$, which yields (ii) as in \refS{SpfT1}.
\end{proof}

We do not know any direct proof of the analogue of
\refL{LB1} for $\tpsifw$. 
(This result follows by \refL{LA2} and \refT{T2}
once the latter is proven.)
However, as in \refSS{SS1A} we nevertheless can use the
following lemma, which is a
strengthening of \refL{LC1}.

\begin{lemma}
  \label{LC2}
Let $F$ be a graph with $e(F)>0$ and let $W$ be a
graphon. If $\tpsifw(\xxf)=p\eff$ for every
$(\xxf)\in\oiff$, then $W=p$.
\end{lemma}

\begin{proof}
As in the proof of \refL{LC1}, first take 
$x_1=\dots=x_\ff=x$. Then
  $\tpsifw(\xxf)=\psifw(\xxf)=W(x,x)\eff$, and thus $W(x,x)=p$.
Using this, it is easy to see that if we
take $x_1=x$ and $x_2=\dots=x_\ff=y$,
and $d_i$ is the degree of vertex $i$, then
\begin{equation*}
  p\eff=\tpsifw(\xxf)=
\frac1{|F|}\sum_{i\in V(F)}
  \parfrac{W(x,y)}{p}^{d_i} p^{e(F)}.
\end{equation*}
Since the \rhs{} is a strictly increasing function of
$W(x,y)$, this equation has only the solution  $W(x,y)=p$.
\end{proof}

As in \refS{SpfT1} there is a companion result where we allow
exceptional null sets.

\begin{lemma}
  \label{LC2ae}
Let $F$ be a graph with $e(F)>0$ and let $W$ be a graphon. 
If $\tpsifw(\xxf)=p\eff$ for \aex{}
$(\xxf)\in\oiff$, then $W=p$ a.e.
\end{lemma}

\begin{proof}
We have not tried to find a direct proof, since this follows directly
from \refL{LC2} and \refC{CD}.
\end{proof}

\refT{T2} now follows from Lemmas
\refand{LA2}{LC2ae}.
(Alternatively, we may use \refL{LD} or
\refT{TD}\ref{tduvs} and argue as in the proof of \refL{LC2}.)

\section{Further variations}\label{Svar}

\subsection{Disjoint subsets}\label{SSdisjoint}
In \refS{SpfT1} the sets $U_1,\dots,U_\ff$ of
vertices were arbitrary and in \refS{SpfT2} they were assumed
to coincide. The opposite extreme is to require that they are disjoint.
We can translate this version too to graphons as follows. Note that (iii)
in the following lemma is that same as \refL{LA1}(iii); hence
the two lemmas together show that it is equivalent to assume
\eqref{la1i} (or \eqref{t1}) for disjoint
$\uuf$ only; this implies the general case.

\begin{lemma}
  \label{LA1d}
Suppose that $G_n\to W$ for some graphon $W$ and let
$F$ be a fixed graph and $\gax\ge0$ a fixed number.
Then the following are equivalent:
\begin{romenumerate}
  \item
For all  disjoint subsets $\uuf$ of\/ $V(G_n)$, 
\begin{equation*}
  N(F,G_n;\uuf)=\gax\prodif|U_i|+o\bigpar{|G_n|^\ff}.
\end{equation*}
\item
For all disjoint subsets $\aaf$ of\/ $\oi$,
\begin{equation*}
  \intaafx \psifw(\xxf)=\gax\prodif\gl(A_i).
\end{equation*}
\item
$\psifw(\xxf)=\gax$ for \aex{} $\xxf\in\oiff$.
\end{romenumerate}
\end{lemma}

\begin{proof}
  Again we follow the proof of \refL{LA1}. The only
  difference is that we consider only disjoint sets $\uuf$, etc.
In particular, given disjoint subsets $\aaf$ of
  $\oi$, we want to construct the random sets $B_i$ so
  that they too are disjoint. We do this by taking, for each
  $j$, the 0--1 random variables $J_{ij}$ dependent, so
  that $\sum_i J_{ij} \le1$. (This is possible
  because $\sum_i a_{ij}\le1$ when $\aaf$
  are disjoint.)
The vectors $(J_{ij})_{i=1}^\ff$ for different $j$
  are chosen independent as before.
Just as in the proof of \refL{LA2}, the dependency among the $J_{ij}$
means that \eqref{sofie} is not exact:
in analogy with \eqref{magnus},
$\E\prodif J_{ij_i}= \prodif a_{ij_i}$
when $j_1,\dots,j_\ff$ are distinct, but not in general.
However, again as in the proof of \refL{LA2}, the total error
  is $o(1)$, so the analogue of \eqref{cecilia} holds, and thus the
  conclusion 
$\int_{\aafx} (\psifwn-\gax)=o(1)$
of \eqref{sofie} holds for all disjoint sets $\aaf$.

Finally, for (ii)$\implies$(iii), note that 
\refL{L1} already is stated so that it suffices to
consider disjoint $\aaf$.
\end{proof}

\refL{LA1d},
combined with the remainder of the proof of \refT{T1} in 
\refS{SpfT1}, shows that in \refT{T1},
it is sufficient to assume \eqref{t1} for disjoint $\uuf$.

\subsection{Sets of the same size}\label{SSsamesize}

Another variation of \refT{T1} is to consider only subsets
$\uuf$ of the same size. (We may combine this with the preceding
variation and require that the sets are disjoint too.)
This can be 
translated to considering only subsets $\aaf$ of the same measure by
the same method as in the next subsection, when we further let the
common size be a given number. Since we obtain stronger results in the
next subsection, we leave the details to the reader.

\subsection{Sets of a given size}\label{SSgivensize}

Another variation of \refT{T1} is 
\refT{T1c} where we
consider only subsets
$\uuf$ of a given size, which we assume is a fixed fraction $\gc$ of
$|G_n|$ (rounded to an integer).
This is translated to graphons as follows.

\begin{lemma}
  \label{LA1c}
Suppose that $G_n\to W$ for some graphon $W$ and let
$F$ be a fixed graph and $\gax\ge0$ and $\gc\in(0,1)$ be fixed numbers.
Then the following are equivalent:
\begin{romenumerate}
  \item\label{La1ci}
For all subsets $\uuf$ of\/ $V(G_n)$
with $|U_i|=\floor{\gc|G_n|}$,
\begin{equation}\label{la1ci}
  N(F,G_n;\uuf)=\gax\prodif|U_i|+o\bigpar{|G_n|^\ff}.
\end{equation}
\item\label{La1cii}
For all subsets $\aaf$ of\/ $\oi$ with $\gl(A_i)=\gc$,
\begin{equation}\label{la1cii}
  \intaafx \psifw(\xxf)=\gax\prodif\gl(A_i).
\end{equation}
\item\label{La1ciii}
$\psifw(\xxf)=\gax$ for \aex{} $\xxf\in\oiff$.
\end{romenumerate}

If $\gc<1/\ff$, we may further, as in \refL{LA1d}, in \ref{La1ci}
and \ref{La1cii} add the requirement that the sets be
disjoint.
\end{lemma}

\begin{proof}
  The equivalence  \ref{La1ci}$\iff$\ref{La1cii} is
  proved as in the proof of \refL{LA1}, but some care has to
  be taken with the sizes and measures of the sets.
We note that for any sets $\aaf$ and $\aayf$, 
\begin{equation}\label{adiff}
\biggabs{\int_{\aafx}
  \psifwn-\int_{\aafyx}\psifwn} 
\le \sumif\gl(A_i\setdiff A'_i).
\end{equation}
Hence, we can modify the sets without affecting the results
as long as the difference has measure $o(1)$. We argue as
follows.

We obtain as in \refS{SpfT1} that \ref{La1ci} is
equivalent to \eqref{e1'}, now for 
all subsets $U_i''$ of\/ $\oi$ that are unions of sets $\inj''$
and have measures
$\gl(U_i'')=\floor{\gc\gnabs}/\gnabs$.
If \ref{La1cii} holds, we may for any such $U_i''$ find
$A_i\supseteq U_i''$ with $\gl(A_i)=\gc$; 
then \eqref{la1cii} implies first
\eqref{e1''} and then  \eqref{e1'} by
\eqref{adiff}.

Conversely, given $\aaf$ with measures
$\gl(A_i)=\gc$, the random sets $B_i$ constructed
above (either as in \refS{SpfT1} or as in
\refSS{SSdisjoint} in the disjoint case)
have measures that are random but well concentrated:
\begin{align*}
  \E\gl(B_i)
&=\sum_j\E J_{ij}\gl(\injy)
=\sum_j a_{ij}\gl(\injy)
=\gl(A_i)=\gc
\\
  \Var\gl(B_i)
&=\sum_j\Var (J_{ij})\gl(\injy)^ 2
\le\gnabs\qw\to0.
\end{align*}
Hence, if $\gd_n\=\gnabs^{-1/3}$, say, then by
Chebyshev's inequality
\begin{equation*}
  \P(|\gl(B_i)-\gc|>\gd_n)
\le \gd_n\qww\Var(\gl(B_i))\le\gd_n\to0.
\end{equation*}
If $|\gl(B_i)-\gc|\le\gd_n$ for all $i$, we
adjust $B_i$ to a set $U_i''$ with
$\gl(U_i'')=\floor{\gc\gnabs}/\gnabs$ so
that $\gl(B_i\setdiff
U_i'')\le\gd_n+\gnabs\qw\le2\gd_n$, and thus 
$$\int_{\bbfx}\psifwn = \int_{\uuiifx}\psifwn + O(\gd_n).$$
Consequently, if \eqref{e1'} holds, then
$\int_{\bbfx}\psifwn = \gax\gc^\ff + O(\gd_n)+o(1)$ whenever
$|\gl(B_i)-\gc|\le\gd_n$ for all $i$, and thus
\begin{equation*}
  \begin{split}
  \E\int_{\bbfx}\psifwn 
&
= \gax\gc^\ff + O(\gd_n)+o(1)+
O\Bigpar{\sumif\P\bigpar{|\gl(B_i)-\gc|>\gd_n}}
\\&
=\gax\gc^\ff+o(1).	
  \end{split}
\end{equation*}
Hence, \eqref{e1''} holds, for $\aaf$ with measures
$\gl(A_i)=\gc$, and thus \ref{La1cii} holds by
the argument in \refS{SpfT1}.

This proves \ref{La1ci}$\iff$\ref{La1cii}; we may add the
requirement that the sets be disjoint by the argument in the proof of
\refL{LA1d}. 

To see that \ref{La1cii}$\iff$\ref{La1ciii}, we use
the following analysis lemma. (This seems to be less well-known that
\refL{L1}; we guess that it is known, but we have been unable
to find a reference.)
\end{proof}

\begin{lemma}
  \label{L3}
Let $\gc\in(0,1)$.
Suppose that $f:\oi^m\to\bbR$ is an integrable function
such that $\intaamx f=0$ for all sequences $\aam$ of
measurable subsets of\/ $\oi$ such that
$\gl(A_1)=\dots=\gl(A_m)=\gc$.
Then $f=0$ \aex

Moreover, if $\gc<m\qw$, it is enough to consider
disjoint $\aam$.
\end{lemma}

\begin{proof}
For $f\in L^1(\oi^m)$ and $A_1,\dots,A_m\subseteq\oi$, let
\begin{equation*}
  f(A_1,\dots,A_m)\=\intaamx f,
\end{equation*}
and define further the functions 
\begin{equation*}
  f_{A_1}(x_2,\dots,x_m)\=\int_{A_1}f(x_1,x_2,\dots,x_m)\dd x_1
\end{equation*}
and
\begin{equation*}
  \fAAm(x_1)\=\int_{\AAmx}f(x_1,x_2,\dots,x_m)\dd x_2\dotsm\dd x_m.
\end{equation*}
By Fubini's theorem,
\begin{equation}
  \label{vattholm}
f(\aam)=f_{A_1}(\AAm)=\fAAm(A_1).
\end{equation}

We will derive the lemma from the following 
claims, which we will prove by induction in $m$.

\emph\bgroup
  Let $B$ be a measurable subset of\/ $\oi$, let
  $0<\gc<1$ and let $f$ be an integrable function
  on $B^{m}$.
\begin{romenumerate}
  \item
If $\gc<\gl(B)$ and 
$f(\aam)=0$ for all $\aam\subset B$ with
$\gl(A_1)=\dots=\gl(A_m)=\gc$, then
$f(\aam)=0$ for all $\aam\allowbreak\subseteq B$.
  \item
If $m\gc<\gl(B)$ and 
$f(\aam)=0$ for all disjoint $\aam\subset B$ with
$\gl(A_1)=\dots=\gl(A_m)=\gc$, then
$f(\aam)=0$ for all disjoint $\aam\subset B$
with $\gl(A_1),\dots,\gl(A_m)\le\gc$.
\end{romenumerate}
\egroup

Consider first the case $m=1$, in which case (i) and (ii)
have the same hypotheses: $\gc<\gl(B)$ and $f(A)=0$ if $\gl(A)=\gc$.
Suppose that $A_1,A_2\subset B$ with
$\gl(A_1)=\gl(A_2)\le\gd:=\tfrac12(\gl(B)-\gc)$.
Then $\gl(B\setminus(A_1\cup A_2))\ge\gl(B)-2\gd=\gc$, and we
may thus find a set $A_0\subseteq B\setminus(A_1\cup
A_2)$ with $\gl(A_0)=\gc-\gl(A_1)$. The assumption
yields $f(A_1\cup A_0)=0=f(A_2\cup A_0)$,
and thus 
\begin{equation}\label{aa12}
f(A_1)=-f(A_0)=f(A_2).  
\end{equation}
If $A\subset B$ is given with $\gl(A)\le\gd$
and $\gl(A)=\gc/N$ for some integer $N$, let
$A_1=A$ and choose further sets
$A_2,\dots,A_N\subset B$ of the
same measure $\gc/N$ and with $A_1,\dots,A_N$
disjoint. By \eqref{aa12}, then $f(A_k)=f(A_1)=f(A)$ for
every $k\le N$, and thus, by the assumption,
\begin{equation*}
  0=f\Bigpar{\bigcup_{k=1}^N A_k} = \sum_{k=1}^Nf(A_k)=Nf(A).
\end{equation*}
Consequently, $f(A)=0$ for every 
$A\subset B$ with $\gl(A)\le\gd$
and $\gl(A)=\gc/N$. If $x_0$ is a density point of
$B$ (\ie, a point in $B$ that is a Lebesgue point of
$\etta_B$), then there is a sequence $\eps_n\to0$
such that $\gl(B\cap(x_0-\eps_n,x_0+\eps_n))=\gc/n$,
and thus by we just have shown,
$\int_{x_0-\eps_n}^{x_0+\eps_n} f\etta_B=
f(B\cap(x_0-\eps_n,x_0+\eps_n))=0$ for every $n$. If
further $x_0$ is a Lebesgue point of $f\etta_B$, then this
implies $f(x_0)=f(x_0)\etta_B(x_0)=0$. Since \aex{}
$x_0\in B$ satisfies these conditions, $f=0$
\aex{} on $B$, which of course is equivalent to
$f(A)=0$ for every $A\subseteq B$.
This proves both (i) and (ii) for $m=1$.

For $m>1$, we use, as already said, induction, and assume that
the claims are true for smaller $m$.
To prove (i), we fix $A_1\subset B$ with
$\gl(A_1)=\ga$, 
and see by \eqref{vattholm}
that $f_{A_1}$ satisfies the assumptions of (i) on
$B^{m-1}$. Thus, by the induction hypothesis,
$f_{A_1}(\AAm)=0$ for all $\AAm\subseteq B$.
Fixing now instead such $\AAm$, \eqref{vattholm}
shows that $\fAAm(A_1)=0$ for all $\gl(A_1)\subset B$ 
with $\gl(A_1)=\gc$, and thus by the case
$m=1$, 
$\fAAm(A_1)=0$ for all $\gl(A_1)\subset B$.
By \eqref{vattholm} again, this proves the induction hypothesis.
Thus (i) is proved in general.

To prove (ii), we again fix $A_1$, and see by \eqref{vattholm}
that $f_{A_1}$ satisfies the assumptions of (ii) on
$(B\setminus A_1)^{m-1}$, noting that
$(m-1)\gc<\gl(B\setminus A_1)$.
Thus, by the induction hypothesis,
$f_{A_1}(\AAm)=0$ for all disjoint $\AAm\subseteq B\setminus A_1$ with
$\gl(A_k)\le\gc$ for every $k$.
Hence, if we instead fix disjoint sets $\AAm\subset B$
with $\gl(A_k)\le\gc$ for every $k$, then
\eqref{vattholm}
shows that $\fAAm(A_1)=0$ for every $A_1\subset
B\setminus(A_2\cup\dots\cup A_m)$ 
with $\gl(A_1)=\gc$, and thus by the case
$m=1$, 
$\fAAm(A_1)=0$ for every $A_1\subset B\setminus(A_2\cup\dots\cup
A_m)$ 
with $\gl(A_1)\le\gc$.
By \eqref{vattholm} again, this proves the induction hypothesis, and
(ii) is proved. 

We have proved the claims above. We now take $B=\oi$ and
the lemma follows immediately by  \refL{L1}.
\end{proof}

\begin{remark}
\label{Rl3}  
When $\gc=m\qw$, it is not enough to consider disjoint sets $\aam$
in \refL{L3}. In
fact, any $f$ of the type $\sumim g(x_i)$ where $\intoi g=0$ satisfies
the assumption for such $\aam$. 
(We do not know whether these are the only possible $f$.)
Taking $W$ of this type and
$F=K_2$, so that $\psifw=W$, 
we get a counter-example to \refL{LA1c}, 
and to \refT{T1c},
for disjoint sets and $\gc=1/\ff$; see also
\refS{Scut} where this example reappears in a different formulation.
We do not know whether there are such counter-examples for other
graphs $F$.
\end{remark}

\begin{proof}[Proof of \refT{T1c}]
  \refT{T1c} follows by using \refL{LA1c} instead of \refL{LA1} in 
(any version of) the
  proof of \refT{T1} in \refS{SpfT1}.
\end{proof}

\subsection{A single subset of a given size}

The corresponding variation of \refT{T2} is \refT{T2c} where we
consider a single subset $U$ with a given 
fraction $\gc$ of the vertices.
Again, there is a straightforward translation to graphons.

\begin{lemma}
  \label{LA2c}
Suppose that $G_n\to W$ for some graphon $W$ and let
$F$ be a fixed graph and $\gax\ge0$ and $\gc\in(0,1)$ be fixed numbers.
Then the following are equivalent:
\begin{romenumerate}
  \item\label{La2ci}
For every subset $U$ of\/ $V(G_n)$
with $|U|=\floor{\gc|G_n|}$,
\begin{equation*}
  N(F,G_n;U)=\gax|U|^\ff+o\bigpar{|G_n|^\ff}.
\end{equation*}
\item\label{La2cii}
For every subset $A$ of\/ $\oi$ with $\gl(A)=\gc$,
\begin{equation*}
  \int_{A^\ff} \psifw(\xxf)=\gax\gl(A)^\ff.
\end{equation*}
\item\label{La2ciii}
$\tpsifw(\xxf)=\gax$ for \aex{} $\xxf\in\oiff$.
\end{romenumerate}
\end{lemma}

\begin{proof}
The equivalence (i)$\iff$(ii) is proved as for \refL{LA1c},
  using single sets $U$, $A$ and $B$ as in the proof of \refL{LA2}.

The equivalence (ii)$\iff$(iii) follows by the following lemma, which
strenthens \refL{L2} by considering subsets of a given size only.
\end{proof}

\begin{lemma}
  \label{L4}
Let $\gc\in(0,1)$.
Suppose that $f:\oi^m\to\bbR$ is an integrable function
such that $\int_{A^m} f=0$ for all 
measurable subsets $A$ of $\oi$ with
$\gl(A)=\gc$.
Then $\tf=0$ \aex
\end{lemma}

\begin{proof}
We begin by showing that the vanishing property extends to sets
$A$ with measure greater than $\gc$ as follows:
\begin{equation}
  \label{aa}
\text{If $A\subseteq\oi$ with $\gl(A)=r\gc$ for some
rational $r\ge1$, then} \int_{A^m}f=0.
\end{equation}
(The restriction to rational $r$ may easily be removed by continuity,
but it will suffice for us.)
To see this, let $N$ be an integer such that $M\=rN$ is
an integer, and partition $A$ into $M$ subsets
$A_1,\dots,A_M$ of equal measure
$\gl(A_i)=\gl(A)/M=r\gc/M=\gc/N$. Pick $N$ of
the sets $A_i$ at random (uniformly over all $\binom
MN$ possibilities), and let $B$ be their union. Thus
$B$ is a random subset of $\oi$ with
$\gl(B)=\gc$, and thus by the assumption
$\int_{B^m}f=0$. Taking the expectation we find
\begin{equation}\label{wattholma}
  0=\E\int_{B^m}f
=\sum_{i_1,\dots,i_m=1}^M\P(A_{i_1},\dots,A_{i_m}\subseteq B) 
\intaaiimx f.
\end{equation}
If $\iim$ are distinct, then, letting $\fallm{N}$ denote the
falling factorial,
\begin{equation*}
  \P(A_{i_1},\dots,A_{i_m}\subseteq B) =
  \frac{\fallm{N}}{\fallm{M}}
=\parfrac NM^m+O\parfrac 1N = r^{-m}+O\parfrac 1N.
\end{equation*}
This fails if two or more of $\iim$ coincide (in fact, the
probability is $\fall N\nu/\fall M\nu\approx
r^{-\nu}$, where $\nu$ is the number of distinct
indices among $\iim$), so we let
$U_N\subseteq\oi^m$ be the union of all $\aaiimx$
with at least two coinciding indices. 
By \eqref{wattholma},
\begin{equation*}
  \begin{split}
    \frac{\fallm{N}}{\fallm{M}}\int_{A^m}f
&= \sum_{\iim=1}^M  \frac{\fallm{N}}{\fallm{M}} \intaaiimx f
\\&
= \sum_{\iim=1}^M 
\lrpar{\frac{\fallm{N}}{\fallm{M}}-\P(\aaiimx\subseteq B)} \intaaiimx f,	
  \end{split}
\end{equation*}
and thus
\begin{equation}\label{vattholma}
\lrabs{ \frac{\fallm{N}}{\fallm M}\int_{A^m}f}
\le \int_{U_N}|f|.
\end{equation}

Now let $N\to\infty$ (with $rN$ integer). 
Note that $\gl(U_N)\le \binom m2 N^{m-1}(\ga/N)^m\le\binom m2/N$.
Thus
$\gl(U_N)\to0$ and hence, since $f$ is integrable, $\int_{U_N}|f|\to0$.
It follows from \eqref{vattholma} and $\fallm{N}/\fallm{M}\to r^{-m}$ that
$r^{-m}\int_{A^m}f=0$, which proves \eqref{aa}.

Next, let $\aam$ be arbitrary disjoint subsets of
$\oi$ with equal measure
$\gl(A_1)=\dots=\gl(A_m)=q\gc$, for some rational
$q$ such that $(1+mq)\gc\le1$. Choose
$A_0\subseteq\oi\setminus\bigcup_1^m A_i$ with $\gl(A_0)=\gc$.
For any  sequence $\xim\in\setoi^m$, let $\xi_0\=1$ and take
$A\=\bigcup_{i\ge0:\xi_i=1} A_i$. Then $\etta_A=\sumiom\xi_i\etta_{A_i}$
and we argue as in the proof of \refL{L2} with an extra set
$A_0$: we have
\begin{equation}
  \label{l4a}
0=\int_{A^m} f = \int_{\oi^m} f\etta_{A^m}
=\sum_{i_1,\dots,i_m=0}^m\xi_{i_1}\dotsm\xi_{i_m}
\int_{A_{i_1}\times\dots\times A_{i_m}} f.
\end{equation}
As in the proof of \refL{L2}, it follows that the coefficient
of $\xi_1\dotsm\xi_m$ in \eqref{l4a} must
vanish,
and this coefficient comes from the terms
where $\iim$ is
a permutation of $1,\dots,m$. We thus obtain
\begin{equation*}
  0=\sum_{\gs\in\fS_m}\int_{\aagsmx} f
=m!\int_{\aamx}\tf.
\end{equation*}
The result follows by Lemma \ref{L1} or \ref{L3}, applied to $\tf$.
\end{proof}

\begin{proof}[Proof of \refT{T2c}] 
  \refT{T2c} follows by combining \refL{LA2c} and \refL{LC2ae}, \cf{}
  \refS{SpfT2}. 
\end{proof}

\section{Induced subgraph counts}

When considering counts of induced subgraphs, we translate the
conditions to graphons similarly as above.

\begin{lemma}
  \label{LAi}
Suppose that $G_n\to W$ for some graphon $W$ and let
$F$ be a fixed graph and $\gax\ge0$ a fixed number.
Then the following are equivalent:
\begin{romenumerate}
  \item\label{LAiU}
For all subsets $\uuf$ of\/ $V(G_n)$, 
\begin{equation*}
  \Nq(F,G_n;\uuf)=\gax\prodif|U_i|+o\bigpar{|G_n|^\ff}.
\end{equation*}
\item\label{LAiA}
For all subsets $\aaf$ of\/ $\oi$,
\begin{equation*}
  \intaafx \psiqfw(\xxf)=\gax\prodif\gl(A_i).
\end{equation*}
\item
$\psiqfw(\xxf)=\gax$ for \aex{} $\xxf\in\oiff$.
\end{romenumerate}
We may further in \ref{LAiU} and \ref{LAiA} add the
conditions that, as in \refL{LA1d}, the sets be disjoint, or
that, as in \refL{LA1c},
$|U_i|=\floor{\gc|G_n|}$ and
$\gl(A_i)=\gc$ for a fixed $\gc\in(0,1)$, or,
provided $\ga<1/\ff$, both.
\end{lemma}

\begin{proof}
  As for \refL{LA1}, using \eqref{nfguq} instead of \eqref{nfgu},
and with the extra conditions treated as for Lemmas \refand{LA1d}{LA1c}.
\end{proof}

\begin{lemma}
  \label{LAis}
Suppose that $G_n\to W$ for some graphon $W$ and let
$F$ be a fixed graph and $\gax\ge0$ a fixed number.
Then the following are equivalent:
\begin{romenumerate}
  \item\label{LAisU}
For all subsets $U$ of\/ $V(G_n)$, 
\begin{equation*}
  \Nq(F,G_n;U)=\gax|U|^\ff+o\bigpar{|G_n|^\ff}.
\end{equation*}
\item\label{LAisA}
For all subsets $A$ of\/ $\oi$,
\begin{equation*}
  \intaf \psiqfw(\xxf)=\gax\gl(A)^\ff.
\end{equation*}
\item\label{LAisPsi}
$\tpsiqfw(\xxf)=\gax$ for \aex{} $\xxf\in\oiff$.
\end{romenumerate}
We may further in \ref{LAisU} and \ref{LAisA} add the
conditions that, as in \refL{LA2c},
$|U|=\floor{\gc|G_n|}$ and
$\gl(A)=\gc$ for a fixed $\gc\in(0,1)$.
\end{lemma}

\begin{proof}
  As for \refL{LA2}, using \eqref{nfguq} instead of \eqref{nfgu},
and with the extra size conditions treated as for \refL{LA2c},
using \refL{L4}.
\end{proof}

However, it is now more complicated to do the algebraic step, \ie,
to solve the equations in (iii) in these lemmas; the reason is that
$\psiqfw$  and $\tpsiqfw$ 
are not monotone in $W$. For $\psiqfw$, we can argue as
follows. (See also the somewhat different argument in \cite{ShapiraY}.) 

\begin{lemma}
  \label{LI}
Let $F$ be a graph with $\ff>1$, let $W$ be a graphon and let $p\in(0,1)$. 
If $\psiqfw(\xxf)=p^{e(F)}(1-p)^{\binom{\ff}2-e(F)}$ for 
every $\xxf\in\oiff$, then either $W=p$ or $W=\cpp$.
\end{lemma}

\begin{proof}
  First, take all $x_i$ equal. Recalling the definitions
  \eqref{psiqfw} and \eqref{beta}, we see that
  \begin{equation*}
\psiqfw(x,\dots,x)=W(x,x)^{e(F)}(1-W(x,x))^{e(\cp F)}=\gbf(W(x,x)).
  \end{equation*}
Thus, $\gbf(W(x,x))=\gbfp$, and hence, \cf{}
\refS{Ssub}, $W(x,x)\in\set{p,\cpp}$ for
every $x$. 

Next, if vertex $i$ has degree $d_i$ and we choose
$x_i=y$ and $x_j=x$ for $j\neq i$, then
\begin{equation*}
  \psiqfw(\xxf)=\parfrac{W(x,y)}{W(x,x)}^{d_i}
\parfrac{1-W(x,y)}{1-W(x,x)}^{\ff-1-d_i} \psiqfw(x,\dots,x),
\end{equation*}
and thus
\begin{equation}\label{ck'}
\parfrac{W(x,y)}{W(x,x)}^{d_i}
\parfrac{1-W(x,y)}{1-W(x,x)}^{\ff-1-d_i} =1,
  \qquad i\in V(F).
\end{equation}

If $F$ is not regular, we may choose vertices $i$ and
$j$ with $d_i\neq d_j$. Taking logarithms of
\eqref{ck'} and the same equation with $i$ replaced by
$j$, we obtain a non-singular homogeneous system of linear
equations in $\log(W(x,y)/W(x,x))$ and $\log((1-W(x,y))/(1-W(x,x)))$,
and thus these logarithms vanish, so $W(x,y)=W(x,x)$ for every
$x$ and $y$ in $\oi$. Hence, if
$x,y\in\oi$, then $W(x,x)=W(x,y)=W(y,x)=W(y,y)$, and
it follows that $W$ is constant, and thus either $W=p$ or
$W=\cpp$.

It remains to treat the case when $F$ is regular, $d_i=d$
for all $i$. Note first that if $F$ is a complete graph,
then $\psiqfw=\psifw$, and the result follows by
\refL{LC1}. Further, if $F$ is empty, the result
follows by taking
complements, replacing $F$ by $\cp F$, which is
complete, $W$ by $1-W$, and $p$ by $1-p$.
We may thus assume that $1\le d\le \ff-2$.

We now choose two vertices $i,j\in V(F)$ and let
$x_i=x_j=y$ and $x_k=x$, $k\neq i,j$.
If there is an edge $ij\in E(F)$, then
\begin{multline*}
  \psiqfw(\xxf)
=\parfrac{W(x,y)}{W(x,x)}^{2d-2}
\parfrac{1-W(x,y)}{1-W(x,x)}^{2(\ff-1-d)} 
\\\times
\parfrac{W(y,y)}{W(x,x)}
\psiqfw(x,\dots,x),
\end{multline*}
and thus, using \eqref{ck'},
\begin{equation*}
\frac{W(y,y)}{W(x,x)}
=\parfrac{W(x,y)}{W(x,x)}^{2}
\end{equation*}
or
\begin{equation}\label{kk1}
  W(x,x)W(y,y)=W(x,y)^2.
\end{equation}
Choosing instead $i,j\in V(F)$ with $ij\notin E(F)$,
we similarly obtain 
\begin{equation}\label{kk2}
 (1- W(x,x))(1-W(y,y))=(1-W(x,y))^2.
\end{equation}
Subtracting \eqref{kk2} from \eqref{kk1} we find
\begin{equation*}
  W(x,x)+W(y,y)=2W(x,y)
\end{equation*}
and thus, also using \eqref{kk1} again,
\begin{equation*}
  \begin{split}
\bigpar{W(x,x)-W(y,y)}^2
&= \bigpar{W(x,x)+W(y,y)}^2 -4W(x,x)W(y,y)
\\&
=4W(x,y)^2-4W(x,y)^2=0.	
  \end{split}
\end{equation*}
Hence $W(x,x)=W(y,y)$ for all $x,y\in\oi$, which by
\eqref{kk1} implies that $W$ is a constant, which must be $p$ or $\cpp$.
\end{proof}

As above, the results in the appendix imply that we can relax the
assumption to hold only almost everywhere.
\begin{lemma}
  \label{LIas}
Let $F$ be a graph with $\ff>1$, let $W$ be a graphon and let $p\in(0,1)$. 
If $\psiqfw(\xxf)=p^{e(F)}(1-p)^{\binom{\ff}2-e(F)}$ for 
\aex{} $\xxf\in\oiff$, then either $W=p$ \aex{} or
$W=\cpp$ a.e.
\end{lemma}

\begin{proof}
  By \refC{CD} and \refL{LI}, $W$ has to be a
  constant $c$ a.e.
Then $\psiqfw(\xxf)=\gbf(c)$ \aex, and thus
  $\gbf(c)=\gbf(p)$; hence $c=p$ or $c=\cpp$.
\end{proof}

\begin{proof}[Proof of Theorems \refand{T3}{T3x}]
As in \refS{SpfT1}, we may assume that $G_n\to W$
for some graphon $W$.
By the assumption and \refL{LAi}, then
\begin{equation*}
\psiqfw(\xxf)=\gbfp\=p^{e(F)}(1-p)^{\binom{\ff}2-e(F)}  
\end{equation*}
 for \aex{}
$\xxf\in\oiff$, which by \refL{LIas} implies 
either $W=p$ \aex{} or $W=\cpp$ a.e.
\end{proof}

For $\tpsiqfw$, the situation is even more complicated. In
fact, \citet{SS:ind} showed that the path
$P_3=K_{1,2}$ and its complement $\cp P_3$ are not
\HF{} (recall \refD{DHF}). Thus, the
analogue of \refL{LI} for $\tpsiqfw$ cannot hold in general.

We can, however, easily obtain the partial results of \cite{SS:ind}
by our methods. We note that
by \refT{TD}, it suffices to study 2-type graphons;
equivalently, it suffices to study $\tpsiqfw(\xxf)$ for
sequences $\xxf$ with at most two distinct values.
For any sequence $\xxf$ with $x_i=x$ for
$k$ values of $i$, and $x_i=y$ for the
$\ff-k$ remaining values, we have
\begin{equation}\label{nyk}
  \tpsiqfw(\xxf)=\binom{\ff}{k}\qw Q_k\bigpar{W(x,x),W(y,y),W(x,y)},
\end{equation}
where $Q_k(u,v,s)$ is  the polynomial, defined for a given graph
$F$ and $k=0,\dots,\ff$, by
\begingroup\multlinegap=0pt
\begin{multline*}
  Q_k(u,v,s)=
\\
\sum_{\substack{A\subseteq V(F)\\|A|=k}}
u^{e(A)}(1-u)^{\binom k2-e(A)}
v^{e(\cpA)}(1-v)^{\binom {\ff-k}2-e(\cpA)}
s^{e(A,\cpA)}(1-s)^{k(\ff-k)-e(A,\cpA)},
\end{multline*}
\endgroup
where $\cpA\= V(F)\setminus A$,
$e(A)$ is the number of edges with both endpoints in $A$, and
$e(A,\cpA)$ is the number of edges with one endpoint in
$A$ and one in $\cpA$.

By symmetry, $Q_{\ff-k}(u,v,s)=Q_k(v,u,s)$.
Note that $Q_0(u,v,s)=\gbf(v)$ and $Q_\ff(u,v,s)=\gbf(u)$.
In particular, $Q_0(u,v,s)=\gbf(p)\iff
v\in\set{p,\cpp}$ and $Q_\ff(u,v,s)=\gbf(p)\iff u\in\set{p,\cpp}$

\begin{remark}
These polynomials are essentially the same as the polynomials
$\P^k_{u,v}(s)$ defined by \citet{SS:ind}. More precisely,
\begin{equation*}
  \begin{split}
\P^k_{u,v}(s)
&\=\binom{\ff}k
  u^{e(F)}(1-u)^{e(\cp F)}-Q_k(u,v,s).
  \end{split}
\end{equation*}  
Hence, the condition in \refT{Tind}\ref{TindQ} below is equivalent to
$\P^k_{u,v}(s)=0$, with 
$u,v\in\set{p,\cpp}$. 
\end{remark}

\begin{theorem}\label{Tind}
  Let $F$ be a graph with $|F|>1$ and let
  $0<p<1$. Then the following are equivalent:
  \begin{romenumerate}
\item \label{Tindhf}
$F$ is $\HF(p)$.	
\item \label{Tindpsiae}
If $\psiqfw(\xxf)=\gbf(p)$ for 
\aex{} $\xxf\in\oiff$, then either $W=p$ \aex{} or $W=\cpp$ a.e.
\item \label{Tindpsi}
If $\psiqfw(\xxf)=\gbf(p)$ for 
all $\xxf\in\oiff$, then either $W=p$ or $W=\cpp$
\item \label{TindQ}
If $Q_k(u,v,s)=\binom{\ff}k \gbf(p)$ for
$k=1,\dots,\ff-1$, and
$u,v\in\set{p,\cpp}$, then $u=v=s$.
  \end{romenumerate}
\end{theorem}

\begin{proof}
  \ref{Tindhf}$\iff$\ref{Tindpsiae} follows by
  \refL{LAis} and our general method.

  \ref{Tindpsiae}$\iff$\ref{Tindpsi} follows by
  \refC{CD} (and the comment after it).

  \ref{Tindpsiae}$\iff$\ref{TindQ} follows by
\refT{TD}, together with the remarks on $Q_0$ and
  $Q_\ff$ above.
\end{proof}

\begin{proof}[Proof of \refT{T3y}]
Again we may assume that $G_n\to W$. It then follows by
\refL{LAis}\ref{LAisU}\!\!$\iff$\!\!\ref{LAisPsi}
and \refT{Tind}\ref{Tindhf}\!\!$\implies$\!\!\ref{Tindpsiae}
that either $W=p$ \aex{} or $W=\cpp$ a.e.
\end{proof}

For $F=P_3$, it suffices by symmetry to check
$Q_1$ in \ref{TindQ}; we find
$Q_1(u,v,s)=2vs(1-s)+(1-v)s^2$, and it is easy to find solutions
with $u=v=p\neq s$, see \cite{SS:ind} for details.
On the other hand, \citet{SS:ind} have shown that every regular graph 
(and a few others) satisfies \ref{TindQ}, and thus
is $\HFp$.

The algebraic problem of determining if there are any other cases
where the overdetermined system in \refT{Tind}\ref{TindQ} has a
non-trivial root is still unsolved. 

\section{Cuts}\label{Scut}

\citet{ChungG} considered also $e_G(U,\cpu)$, the number of edges
in the graph $G$
across a cut $(U,\cpu)$,
where $\cpu\=V(G)\setminus U$.
They proved the following results:

\begin{theorem}[\citet{ChungG}]
  \label{Tcut}
Suppose that $(G_n)$ is a sequence of graphs with $|G_n|\to\infty$
and let $0\le p\le1$.
Then $(G_n)$ is \pqr{} if and only if,
for all subsets $U$ of $V(G_n)$,
\begin{equation}\label{tcut}
   \egn(U,\cpu)=p|U||\cpu|+o\bigpar{|G_n|^2}.
\end{equation}
\end{theorem}

\begin{theorem}[\citet{ChungG}]
  \label{Tcutc}
Let $\gc\in(0,1)$ with $\gc\neq1/2$.
Suppose that $(G_n)$ is a sequence of graphs with $|G_n|\to\infty$
and let $0\le p\le1$.
Then $(G_n)$ is \pqr{} if and only if
\eqref{tcut} holds
for all subsets $U$ of $V(G_n)$ with $|U|=\floor{\gc\gnabs}$.
\end{theorem}

However, as shown in \cite{ChungGW:quasi,ChungG}, \refT{Tcutc} does
\emph{not} hold for $\gc=1/2$.

Note that in our notation, 
\begin{equation}
  \label{ecut}
e_G(U,\cpu)=N(K_2,G;U,\cpu),
\end{equation}
 so these
results are closely connected to \refT{T1} and its variants.
We may use the methods above to show these results too, and to see why
$\gc=1/2$ is an exception in \refT{Tcutc}.

We thus assume that $G_n\to W$ for some graphon $W$, and translate the
properties above to properties of $W$.
We state this as a lemma in the same style as earlier, and note that
Theorems \refand{Tcut}{Tcutc}  are immediate consequences.

\begin{lemma}  \label{Lcut1}
Suppose that $G_n\to W$ for some graphon $W$
and let $p\in\oi$.
Then the following are equivalent:
\begin{romenumerate}
  \item\label{Lcut1a}
For all subsets $U$ of\/ $V(G_n)$, 
\begin{equation*}
   \egn(U,\cpu)=p|U||\cpu|+o\bigpar{|G_n|^2}.
\end{equation*}
\item\label{Lcut1b}
For all subsets $A$ of\/ $\oi$,
\begin{equation}
  \label{wcut}
\int_{A\times \cp A} W(x,y) = p\gl(A)\gl(\cp A).
\end{equation}
\item\label{Lcut1c}
$W=p$  \aex{}
\end{romenumerate}
For any fixed $\gc\in(0,1)\setminus\set{\frac12}$,
we may further add the condition that
$|U|=\floor{\gc\gnabs}$ in \ref{Lcut1a} and 
$\gl(A)=\gc$ in \ref{Lcut1b}.
(If we add these conditions with $\gc=1/2$, 
the equivalence \ref{Lcut1a}$\iff$\ref{Lcut1b} still holds,
but these do not imply \ref{Lcut1c}.)
\end{lemma}
\begin{proof}
  The equivalence \ref{Lcut1a}$\iff$\ref{Lcut1b} follows as in
Lemmas \refand{LA1}{LA1d}, arguing as in
\refL{LA1c} in the case of a fixed size $\gc\in(0,1)$.

The implication \ref{Lcut1c}$\implies$\ref{Lcut1b} is trivial, and
\ref{Lcut1b}$\implies$\ref{Lcut1c} follows by the following lemma,
applied to $W-p$. 
\end{proof}

\begin{lemma}
  \label{Lcut}
Let $\ga\in(0,1)\setminus\set{\frac12}$.
If $f:\oi^2\to\bbR$ is a symmetric measurable function such that 
$\int_{A\times(\cpax)}f=0$ for every subset $A$ of $\oi$ with
$\gl(A)=\gc$, then $f=0$ a.e. 
\end{lemma}

\begin{proof}
  Let $f_1(x)\=\int_0^1 f(x,y)\dd y$ be the marginal of $f$. Then
  \begin{equation}\label{fx}
	\begin{split}
0 &= 	
\int_{A\times(\cpax)}f
=\int_A f_1(x)\dd x - \int_{A\times A} f(x,y)\dd x\dd y
\\&
=\int_{A\times A} \Bigpar{\frac1{\gc}f_1(x)-f(x,y)}\dd x\dd y.	  
	\end{split}
  \end{equation}
\refL{L4} now shows that the symmetrization
$\frac1{2\gc}f_1(x)+\frac1{2\gc}f_1(y)-f(x,y)=0$ a.e., \ie, 
\begin{equation}\label{fy}
  f(x,y)=\frac{1}{2\gc}\bigpar{f_1(x)+f_1(y)}.
\end{equation}
Integrating \eqref{fy} with respect to both variables we find
$\int f = \frac{2}{2\gc}\int f$, and thus, because $\gc<1$, $\int f=0$.
Integrating \eqref{fy} with respect to one variables we then find
$f_1(x)=\frac{1}{2\gc}f_1(x)$ \aex, and thus $f_1(x)=0$ \aex{} because
$\gc\neq 1/2$. A final appeal to \eqref{fy} yields $f(x,y)=0$ a.e.
\end{proof}

This proof also shows what goes wrong with  \refT{Tcutc}
when $\gc=1/2$. In this case, 
the condition of \refL{Lcut} still implies
\eqref{fy},
but this is satisfied if (and only if)
$f(x,y)=g(x)+g(y)$ for any integrable
$g$ with $\int g=0$, and as a result we see that \eqref{tcut} is
satisfied for all $U$ with $|U|=\floor{\gnabs/2}$ whenever $G_n\to W$
where $W$ is a graphon of the form $W(x,y)=h(x)+h(y)$ with $\int_0^1 h=p/2$.
(One such example of \gnq{}, with $p=1/2$ and
$h(x)=\frac12\ett{x\ge 1/2}$ is given in \cite{ChungGW:quasi,ChungG}.)
Cf.\ \refR{Rl3}.

\begin{remark}
The condition that $f$ is  symmetric is essential in
\refL{Lcut}.
If $f$ is anti-symmetric, then \eqref{fx} implies
  that $f$ satisfies the condition if and only if $\intoi
  f(x,y)\dd y=0$ for \aex{} $x$. One example is $\sin(2\pi(x-y))$.
\end{remark}

\citet{ChungGW:quasi} remarked that \refT{Tcutc} holds in the case
$\gc=1/2$ too, if we further assume that $(G_n)$ is almost regular 
(see below for definition). We discuss and show this in the next section.

\section{The degree distribution}\label{Sdeg}

If $G$ is a graph, let $D_G$ denote the random variable defined as the degree
$d_v$ of  a randomly chosen vertex $v$ (with the uniform
distribution on $V(G)$). Thus $0\le D_G\le |G|-1$, and we
normalize $D_G$ by considering
$D_G/|G|$, which is a random variable in \oi.
If $(G_n)$ is a sequence of graphs, with $|G_n|\to\infty$ as usual,
we say that $(G_n)$ has \emph{asymptotic (normalized) degree
distribution} $\mu$ if $D_G$ tends to $\mu$ in distribution.
(Here $\mu$ is a distribution, \ie, a probability measure, on $\oi$.)
In the special case when $\mu$ is concentrated at a point $p\in\oi$,
we say that $(G_n)$ is \emph{almost $p$-regular} (or \emph{almost
regular} if we do not want to specify $p$); this thus is the
case if and only if $D_{G_n}\pto p$, with convergence in probability,
which means that all but $o(|G_n|)$ vertices in $G_n$ have degrees
$p|G_n|+o(|G_n|)$. 
Since the random variables $D_{G_n}$ are uniformly bounded (by 1), 
this is further equivalent to convergence in mean, and thus
a sequence  $(G_n)$ is almost $p$-regular if and only if
$\E|D_{G_n}-p|\to0$, or, more explicitly, \cf{} \cite{ChungGW:quasi},
\begin{equation}\label{areg}
\sum_{v\in V(G)} \bigabs{d_v-p|G_n|} =o(|G_n|^2).	  
\end{equation}

The normalized degree distribution behaves continuously under graph
limits, and 
a corresponding ``normalized degree distribution'' may be defined for
every graph limit too. 
(See further \cite{threshold}.)
For a graphon $W$ we define the marginal
$w(x)\=\intoi W(x,y)\dd y$ and the random variable
$D_W\=w(U)=\intoi W(U,y)\dd y$, where $U\sim U\oi$ is uniformly
distributed on $\oi$. 

\begin{theorem}\label{Tdegree}
If $G_n$ are graphs with $|G_n|\to\infty$ and $G_n\to W$ for some
graphon $W$, then $D_{G_n}/|G_n|\dto D_W$.
Hence, $(G_n)$ has an asymptotically degree distribution, and this equals
the distribution of the random variable $D_W\=\intoi W(U,y)\dd y$.
\end{theorem}

\begin{proof}
  It is easily seen that, for every $k\ge1$, the moment $\E (D_G/|G|)^k$ equals
  $t(S_k,G)$, where $S_k=K_{1,k}$ is a star with $k+1$ vertices, and
  similarly the moment $\E W_G^k=t(S_k,W)$.
Consequently, 
$\E (D_{G_n}/|G_n|)^k=t(S_k,G_n)\to t(S_k,W)=\E D_W$ for every
  $k\ge1$, and thus $D_{G_n}\dto D_W$ by the method of moments.
\end{proof}

\begin{corollary}\label{Cdegree}
Let $(G_n)$ be a sequence of graphs and $W$ a graphon such that
  $G_n\to W$.
Then $G_n$ is almost $p$-regular if and only if
$\intoi W(x,y)\dd y=p$ for \aex{} $x\in\oi$.
\nopf
\end{corollary}

In particular, a \qr{} sequence of graphs is almost regular, but 
the converse does not hold.

Motivated by \refC{Cdegree},
we say that a graphon $W$ is \emph{$p$-regular} if its marginal
$\intoi W(x,y)\dd y=p$ a.e.
This is evidently not a \qr{} property of graphons, but it can be used
in conjuction with the failed case $\gc=1/2$ in \refS{Scut}.
We find the following lemmas.

\begin{lemma}
  \label{LDcut}
Let $\ga\in(0,1)$.
If $f:\oi^2\to\bbR$ is a symmetric measurable function such that 
$\int_{A\times(\cpax)}f=0$ for every subset $A$ of $\oi$ with
$\gl(A)=\gc$, and $\intoi f(x,y)\dd y=0$ for \aex{} $x$, then $f=0$ a.e. 
\end{lemma}

\begin{proof}
  The proof of \refL{Lcut} shows that \eqref{fy} holds, where now by
  assumption $f_1=0$.
\end{proof}

\begin{lemma}  \label{LD1}
Let $p\in\oi$ and $\gc\in(0,1)$.
Suppose that $(G_n)$ is an almost $p$-regular sequence of graphs and that
$G_n\to W$ for some graphon $W$.
Then the following are equivalent:
\begin{romenumerate}
  \item\label{LDcut1a}
For all subsets $U$ of\/ $V(G_n)$
with $|U|=\floor{\gc\gnabs}$,
\begin{equation}\label{tcutd}
   \egn(U,\cpu)=p\ga(1-\ga)|G_n|^2+o\bigpar{|G_n|^2}.
\end{equation}
\item\label{LDcut1b}
For all subsets $A$ of\/ $\oi$ with $\gl(A)=\gc$,
\begin{equation*}
\int_{A\times \cp A} W(x,y) = p\gc(1-\gc).
\end{equation*}
\item\label{LDcut1c}
$W=p$  \aex{}
\end{romenumerate}
\end{lemma}
\begin{proof}
By \refL{Lcut1}, it remains only to show that
\ref{LDcut1b}$\implies$\ref{LDcut1c} in the case $\gc=1/2$.
However, by \refC{Cdegree}, $W$ is $p$-regular, 
so \ref{LDcut1b}$\implies$\ref{LDcut1c} follows by
\refL{LDcut} applied to $W-p$.
\end{proof}

\refL{LD1} yields, by our general machinery, immediately the following
theorem by \citet{ChungGW:quasi}, which supplements \refT{Tcutc}
in the case $\gc=1/2$ (and otherwise is a trivial consequence of \refT{Tcutc}).
\begin{theorem}[\citet{ChungGW:quasi}]
  \label{TDcutc}
Let $0\le p\le1$ and $\gc\in(0,1)$.
Suppose that $(G_n)$ is a sequence of graphs with
$|G_n|\to\infty$.
Then $(G_n)$ is \pqr{} if and only if
$(G_n)$ is almost $p$-regular and
\eqref{tcutd} holds
for all subsets $U$ of $V(G_n)$ with $|U|=\floor{\gc\gnabs}$.
\end{theorem}

\appendix

\section{A measure-theoretic lemma}\label{Appa}

A \emph{multiaffine} polynomial is a polynomial in several
variables $\set{x_\nu}_{\nu\in\cI}$, for some
(finite) index set $\cI$, such that each variable has degree at
most 1; it can thus be written as a linear combination of the
$2^{|\cI|}$
monomials $\prod_{\nu\in \cJ} x_\nu$ for subsets $\cJ\subseteq\cI$.
We are interested in the case when the index set $\cI$
consists of the $\binom m2$ pairs $\set{i,j}$
with $1\le i<j\le m$, for some $m\ge2$. In this
case we define, for any symmetric function
$W:\oi^2\to\bbR$ and $x_1,\dots,x_m\in\oi$,
\begin{equation}\label{phiw}
  \phiw(x_1,\dots,x_m)\=\Phi\bigpar{(W(x_i,x_j))_{i<j}}.
\end{equation}

The functions $\psifw$ and $\psiqfw$ considered
above are of this type, see \eqref{psifw} and \eqref{psiqfw},
as well as their symmetrizations
$\tpsifw$ and $\tpsiqfw$. As we have seen above, in
all our proofs we derive as an intermediate result an equation of the
type $\Phi_W(\xxm)=\gax$ \aex{} for some multiaffine
$\Phi$, and it would simplify the
analysis of this equation if we were able to strengthen this
to $\Phi_{W}(\xxm)=\gax$ for \emph{every}
$\xxm\in\oi$, possibly after modifying $W$ on a
null set. We thus are led to the following measure-theoretic problem,
with applications to quasi-random graphs:

\begin{problem}\label{P1}
Suppose that $\Phi\bigpar{(w\ij)_{i<j}}$ is a multiaffine polynomial in the
  $\binom m2$ variables $w\ij$, $1\le  i<j\le m$, for some $m\ge2$.
Suppose further that $W:\oi^2\to\oi$ is a graphon such that
$\Phi_W(\xxm)=\gax$ a.e. for some $\gax\in\bbR$.
Does there always exist a
graphon $W'$ with $W'=W$ \aex{} such that
$\Phi_{W'}(\xxm)=\gax$ for \emph{every} $\xxm\in\oi$?
\end{problem}

We were able to prove such a result for a special class of
$\Phi$ in \refL{LB1} (but see \refR{RB1}). In general, we do not know the
answer, but we can prove the following weaker result that suffices for
us;
the important feature is that the set $E$ below contains the
diagonal; hence we can make the equation 
$\Phi_{W'}(\xxm)=\gax$ hold (typically, at least) also when
several, or all, $x_i$ coincide.

\begin{remark}
The elimination of a null set in \refP{P1} seems related to the infinite
version of the (hypergraph) removal lemma \cite{ElekSz},
where the objective, in a different but related context, also is to
replace a null set by an empty set.  
\end{remark}

\begin{lemma}\label{LD}
  Suppose that $\Phi\bigpar{(w\ij)_{i<j}}$ is a multiaffine polynomial in the
  $\binom m2$ variables $w\ij$, $1\le  i<j\le m$, for some $m\ge2$.
Suppose further that\/ $W:\oi^2\to\oi$ is a graphon,
  \ie, a symmetric measurable function, and 
suppose that $\phiw(\xxm)=\gax$ for \aex{} $\xxm\in\oi$ and some $\gax\in\bbR$.
Then there is a version\/ $W'$ of\/ $W$ and a symmetric set  $E\subseteq\oi^2$
such that $\gl(\oi^2\setminus E)=0$, $E\supseteq\set{(x,x):x\in\oi}$, and 
$\Phi_{W'}(x_1,\dots,x_m)=\gax$ for all
  $x_1,\dots,x_m$ such that $(x_i,x_j)\in E$ for
  every pair $(i,j)$ with $1\le i<j\le m$.
\end{lemma}

The proof is rather technical, and is postponed until the end of the appendix.

As a consequence, we obtain a convenient criterion (patterned after
\cite{SS:ind}). We say that a graphon $W$ is
\emph{finite-type}, or more specifically
\emph{$k$-type}, if there exists a partition of
$\oi$ into $k$ sets $S_1,\dots,S_k$ such that
$W$ is constant on each rectangle $S_i\times S_j$.
Making a rearrangement, we can without loss of generality assume that
the sets $S_i$ are intervals.
(See \cite{LSos} for a study of finite-type graph limits and the
corresponding sequences of graphs, which 
generalize \qr{} graphs.)

\begin{remark}
  In this paper, we consider for convenience only graphons defined on
  $\oi$, but the definition extends to any
  probability space. Using this, we can equivalently, 
and more naturally, say that $W$ is finite-type if it is
equivalent to a graphon defined on a finite probability space.
\end{remark}

\begin{theorem}
  \label{TD}
 Suppose that $\Phi\bigpar{(w\ij)_{i<j}}$ is a multiaffine polynomial in the
  $\binom m2$ variables $w\ij$, $1\le  i<j\le m$, for some $m\ge2$,
  and that $\gax\in\bbR$. Then the following are equivalent.
  \begin{romenumerate}
\item\label{tdw}
There exists a graphon\/ $W$ such that	
$\phiw(\xxm)=\gax$ for \aex{} $\xxm\in\oi$, but\/ $W$ is
\emph{not} \aex{} constant.

\item\label{td2}
There exists a 2-type graphon\/ $W$ such that	
$\phiw(\xxm)=\gax$ for all $\xxm$, but\/ $W$ is
\emph{not} (\aex) constant. 

\item\label{tduvs}
There exist numbers $u,v,s\in\oi$, not all equal, such
that for every subset $A\subseteq[m]$, if we choose
\begin{equation}\label{uvs}
  w\ij\=
  \begin{cases}
	u,&i,j\in A, \\
	v,&i,j\notin A, \\
	s,&i \in A,\,j\notin A \text{ or conversely},
  \end{cases}
\end{equation}
then $\Phi((w\ij)_{i<j})=\gax$.
  \end{romenumerate}
\end{theorem}

In \ref{td2}, we may further require that the two parts of
$\oi$ are the intervals $[0,\frac12]$ and $(\frac12,1]$.

The equivalence of \ref{tdw} and \ref{td2} shows
that if a property of the type $\phiw=\gax$
\aex{} does not imply that $W$ is \aex{}
constant (\ie, it is not a (mixed) \qr{} property for graphons),
then there exists a counter-example that is a 2-type graphon.
This generalizes one of the results for induced subgraph counts by
\citet{SS:ind}.

\begin{proof}
\ref{td2}$\iff$\ref{tduvs}:
A 2-type graphon $W$ is defined by a partition $(S_1,S_2)$
of $\oi$ and three numbers $u,v,s\in\oi$ such
that $W=u$ on $S_1\times S_1$, 
$W=v$ on $S_2\times S_2$, and
$W=s$ on $(S_1\times S_2)\cup(S_2\times S_1)$.
It is easy to see that, for any $S_1$ and $S_2$ with
$\gl(S_1),\gl(S_2)>0$, such a graphon $W$ satisfies
$\phiw=\gax$  if and only if
$\Phi((w\ij)_{i<j})=\gax$ with $w\ij$ as
in \eqref{uvs}, for every choice of
$A\subseteq[m]$. (Consider $x_i$ such that
$x_i\in S_1\iff i\in A$.) 
Moreover, $W$ is  constant $\iff u=v=s$.

\ref{td2}$\implies$\ref{tdw}: Trivial.

\ref{tdw}$\implies$\ref{tduvs}: 
Suppose that $W$ is a graphon as in \ref{tdw} but that
\ref{tduvs} does not hold; we will show that this leads to a contradiction.
Let $W'$ and $E$ be as in \refL{LD}; for notational
simplicity we replace $W$ by $W'$ and assume thus $W'=W$.

Suppose that $(x,y)\in E$. Given $A\subseteq[m]$,
let $x_i\=x$ for $i\in A$ and $x_i\= y$ for
$i\notin A$. Then $W(x_i,x_j)=w\ij$ as given by
\eqref{uvs} with $u=W(x,x)$, $v=W(y,y)$, $s=W(x,y)$.
Further, \refL{LD} shows that
$\Phi((w\ij)_{i<j})=\phiw(\xxm)=\gax$.
Since \ref{tduvs} does not hold, no such $u,v,s$ exist
except with $u=v=s$. Consequently, we have shown the following
property of $W$:
\begin{equation}
  \label{md}
\text{If  $(x,y)\in E$, then } W(x,x)=W(y,y)=W(x,y). 
\end{equation}

Now suppose, more strongly,  that $(x_0,y_0)$ is a Lebesgue point of
$E$, and that $U$ is an open interval with $W(x_0,y_0)\in U$. 
It follows from the definition of Lebesgue points, that in a sufficiently
small square $Q$ centered at $(x_0,y_0)$, the set
$B\=\set{(x,y)\in Q:W(x,y)\in U}$  has measure at least 
$\gl(Q)/2$. Since $\gl(E)=1$, the same holds for
$B\cap E$, and we may thus, by the regularity of the Lebesgue
measure, find a compact set $K\subseteq B\cap E$ with $\gl(K)>0$.
If $(x,y)\in K$, then $(x,y)\in E$, so by
\eqref{md}, $W(x,x)=W(x,y)\in U$. 
Consequently, if $K'$ is the projection of $K$ onto the
first coordinate, then $W(x,x)\in U$ for $x\in K'$;
furthermore, $K'$ is a compact, and thus measurable, subset of
$\oi$, and $\gl(K')>0$.

By assumption, our $W$ is not \aex{} constant. Thus
there exist two disjoint open intervals $U_1$ and $U_2$
such that $W\qw(U_\ell)\=\set{(x,y):W(x,y)\in
U_\ell}\subseteq\oi^2$ has positive measure,
$\ell=1,2$. Then also, for each $\ell=1,2$, 
$D_\ell\=E\cap W\qw(U_\ell)$ has positive
measure, so we may pick a Lebesgue point $(x_\ell,y_\ell)$ in
$D_\ell$.
By what we just have shown, this implies that there exists a compact set
$K_\ell\subseteq\oi$ with $\gl(K_\ell)>0$ and
$W(x,x)\in U_\ell$ for $x\in K_\ell$.

However, this means that if $(x,y)\in K_1\times K_2$, then
$W(x,x)\neq W(y,y)$, and thus by \eqref{md},
$(x,y)\notin E$. Hence $E\cap(K_1\times
K_2)=\emptyset$. Since $\gl(K_1\times K_2)>0$ and
$\gl(E)=1$, this is a contradiction.
\end{proof}

\begin{corollary}  \label{CD}
 Suppose that $\Phi\bigpar{(w\ij)_{i<j}}$ is a multiaffine polynomial in the
  $\binom m2$ variables $w\ij$, $1\le  i<j\le m$, for some $m\ge2$,
  and that $\gax\in\bbR$. 
If every  graphon\/ $W$ such that	
$\phiw(\xxm)=\gax$ for every $\xxm\in\oi$ is constant,
then
every  graphon\/ $W$ such that	
$\phiw(\xxm)=\gax$ for \aex{} $\xxm\in\oi$ is \aex{} constant.
\end{corollary}

In the terminology of \refR{Rqrw}, if 
``$\phiw(\xxm)=\gax$ everywhere'' is a (mixed) \qr{}
property, then so is
``$\phiw(\xxm)=\gax$ \aex''.
It is easily seen that the converse holds too; if $W$ is a
non-constant graphon such that
$\phiw(\xxm)=\gax$ for every $\xxm\in\oi$, then there exists a
non-constant $m$-type graphon with this property, and this
graphon is not \aex{} constant.

\begin{proof}
  The assumption implies that there is no 2-type graphon $W$ as
  in \refT{TD}(ii), and thus there is no graphon $W$ as
  in \refT{TD}(i).
\end{proof}

It remains to prove \refL{LD}. 
In order to do this,
we first prove the following lemma, which is a
(weak) substitute for the Lebesgue differentiation theorem when we
consider points on the diagonal only. (The Lebesgue differentiation theorem
says nothing about such points, since the diagonal is a null set. A
simple counter-example is $W(x,y)=\ett{x<y}$.)
We introduce  some further notation.

If $A\subseteq\oi$ with $\gl(A)>0$, let
$\gl_A$ be the normalized Lebesgue measure on $A$ given
by $\gl_A(B)\=\gl(A\cap B)/\gl(A)$,
$B\subseteq\oi$.
(In other words, $\gl_A$ is the distribution of a uniform random point
in $A$.)

The definition \eqref{cut} of the cut norm generalizes to
arbitrary measure spaces. In particular, if
$A\subseteq\oi$ with $\gl(A)>0$, we let
$\cnx WA$ denote the cut norm on $A\times A$ with respect to the normalized
measure $\gl_A$. More generally, if
$A$ and $B\subseteq\oi$ have positive measures, then 
\begin{equation*}
  \cnx{W}{A\times B}
\=
\sup_{S\subseteq A,\, T\subseteq B}
\int_{S\times T} W(x,y)\dd\gl_A(x)\dd\gl_B(y)
\end{equation*}
denotes the (normalized) cut norm on $A\times B$.

\begin{lemma}\label{LLD}
  For every $\eps>0$ there exists
  $\gd=\gd(\eps)>0$ such that if
  $W:\oi^2\to\oi$ is a symmetric and measurable
  function and $A\subseteq\oi$ with $\gl(A)>0$,
  then there exists $B\subseteq A$ with
  $\gl(B)\ge\gd\gl(A)$ and a real number
  $w\in\oi$ such that $\cnx{W-w}{B}<\eps$.
\end{lemma}

\begin{remark}
  The example $W(x,y)=\ett{x<y}$ shows that \refL{LLD}
  in general fails for non-symmetric functions.
\end{remark}

\begin{remark}
\refL{LLD} is not true with the stronger conclusion obtained by
  replacing cut norm by $L^1$ norm.
An example is (whp) given, for any $\eps<1/2$, by the $0/1$-valued function $W$
corresponding to a random graph $G(n,1/2)$, for a large $n$.
\end{remark}

Although \refL{LLD} is a purely analytic statement, we prove it using
combinatorial methods; in fact, the proof is an adaption of the
relevant parts of the proof of one of the main theorems in
\citet{SS:ind} to graphons (instead of graphs).

\begin{proof}
  By considering the restriction of $W$ to $A\times A$
  and a measure preserving bijection of $(A,\gl_A)$ onto
  $(\oi,\gl)$, it suffices to consider the case
  $A=\oi$.

Let $r=\ceil{3/\eps}$ and let $M$ be the
Ramsey number $R(r;r)=R(r,\dots,r)$ (with $r$ repeated
$r$ times); in other words, every colouring of the edges of the
complete graph $K_M$ with at most $r$ colours contains a
monochromatic $K_r$. (See \eg{} \cite{GRS:Ramsey}.)

By the (strong) analytic \SRL{}  by 
\citet[Lemma 3.2]{LSz:Sz}, there is an integer
$K=k(\eps/(4M^2))$ (depending on $\eps$ only, since
$M$ is a function of $\eps$) and, for some
$k\le K$, a partition
$\cP=\set{S_1,\dots,S_k}$ of $\oi$ into
$k$ sets of equal measure $1/k$ with the property that for
every set $R\subseteq\oi^2$ that is a union of at most
$k^2$ rectangles, we have
\begin{equation}
  \label{ow1}
\lrabs{\int_R(W-\wp)}\le\frac{\eps}{4M^2},
\end{equation}
where $\wp$ is the function that is constant on each set
$S_i\times S_j$ and equal to the average
$k^2\int_{S_i\times S_j} W$ there.
(I.e., $\wp$ is the conditional expectation of $W$
given the $\gs$-field generated by
$\set{S_i\times S_j}_{i,j=1}^k$.)
Let $w_{ij}$ be this average $k^2\int_{S_i\times S_j} W$.
We consider two cases separately:

\pfitem{i} $k\ge2M$.
Let, for $i,j=1,\dots,k$, 
\begin{equation}
  \label{ow2}
d\ij\=\cnx{W-\wp}{S_i\times S_j}
=\cnx{W-w\ij}{S_i\times S_j}
=\max\bigpar{d\ij^+,d\ij^-},
\end{equation}
where
\begin{equation*}
  d\ij^\pm
\=
\sup_{S\subseteq S_i,\; T\subseteq S_j} \pm k^2\int_{S\times T} 
\xpar{W-\wp}
.\end{equation*}
It follows from \eqref{ow1} that
\begin{equation*}
  \sum_{i,j=1}^kd\ij^+\le k^2\frac{\eps}{4M^2},
\end{equation*}
and thus the number of pairs $(i,j)$ with
$d\ij^+>\eps/3$ is less than $k^2/M^2$, and
similarly for $d\ij^-$.

Say that a pair $(i,j)$ is \emph{bad} if
$d\ij>\eps/3$ or $i=j$, and \emph{good}
otherwise.
By \eqref{ow2}, the number of bad pairs is thus less than
$2k^2/M^2+k\le k^2/M$, using our assumption that $k\ge2M$
and assuming, as we may, that $M\ge4$. 

Consider the graph $H$ on $[k]$ where there is an edge
$ij$ whenever $(i,j)$ is a good pair. Further, give every
edge $ij$ in $H$ the colour $c\ij\=\max(\ceil{rw\ij},1)\in[r]$.
Since $H$ has more than $\frac12(k^2-\frac1M
k^2)=(1-\frac1M)\frac{k^2}2$ edges, Tur\'an's theorem
shows that $H$ contains a complete subgraph $K_M$,
and the choice of $M$ implies that this complete subgraph
contains a complete monochromatic subgraph $K_r$.

In other words, there is a $c\in[r]$ such that, after
renumbering the sets $S_i$ in $\cP$, for all
$i,j\in[r]$ with $i\neq j$, $(i,j)$ is a good
pair and $c\ij=c$.
Let $w\=c/r\in\oi$. Then, for $1\le i<j\le r$,
$c-1\le rw\ij\le c$, so $|w\ij-w|\le 1/r\le
\eps/3$.
Since $(i,j)$ is good, this further implies
\begin{equation*}
\cnx{W-w}{S_i\times S_j} \le d\ij+|w\ij-w|\le 2\eps/3,
\qquad 1\le i<j\le r.
\end{equation*}
On the other hand, trivially, for every $i$, 
\begin{equation*}
\cnx{W-w}{S_i\times S_i} \le \sup|W-w|\le1.
\end{equation*}
Let $B\=\bigcup_{i=1}^r S_i$. Then
$\gl(B)=r/k\ge r/K$ and, recalling that the sets
$S_i$ have the same measure,
\begin{equation*}
  \begin{split}
\cnx{W-w}{B} 
&\le 
r\qww\sum_{i,j=1}^r\cnx{W-w}{S_i\times S_j} 
\le r\qww\Bigpar{r(r-1)\frac{2\eps}3+r\cdot1}
\\&
<\frac{2\eps}3+\frac1r\le\eps.	
  \end{split}
\end{equation*}

\pfitem{ii} $k<2M$.
We simple take $B=S_1$ and $W=w_{11}$.
Then $\gl(B)=1/k>1/(2M)$, and \eqref{ow1} implies
\begin{equation*}
\cnx{W-w}{B} 
\le\gl(B)\qww\cn{W-w} 
\le\gl(B)\qww\frac{\eps}{4M^2}
<\eps.
\end{equation*}
This completes the proof of \refL{LLD}.
\end{proof}

\begin{proof}[Proof of \refL{LD}]
We may assume that $\gax=0$.

For $\eps>0$ and $\eta>0$, let
  \begin{multline}
  \label{een}
E_{\eps,\eta}\=
\\
\Bigset{(x,y)\in(0,1)^2:(2\eps)\qww\int_{|x'-x|,|y'-y|<\eps}
|W(x',y')-W(x,y)|\dd x'\dd y'<\eta}.
  \end{multline}
The Lebesgue differentiation theorem says that 
\aex{} $(x,y)\in\bigcap_\eta\bigcup_\eps
E_{\eps,\eta}$; in other words, \aex{}
$(x,y)\in E_{\eps,\eta}$ for every $\eta>0$
and all sufficiently small $\eps>0$ (depending on $x$,
$y$ and $\eta$). For $\eta>0$ and
$n\ge1$, we can thus find
$\eps_1=\eps_1(\eta,n)\in(0,1/n)$ such that
$\gl\bigpar{E_{\eps_1(\eta,n),\eta}}>1-2^{-n}$.

For $n\ge1$, let $\gdn\=\gd(1/n)$ be as in
\refL{LLD} with $\eps=1/n$, and let $\etan\=\gdn^2/n$,
$\eps_2(n)\=\eps_1(\etan,n)$ and $E_n\=E_{\eps_2(n),\etan}$.
Then $\gl\bigpar{E_n}>1-2^{-n}$, so if
$\tE\=\bigcup_{n=1}^\infty\bigcap_{\ell=n}^\infty E_\ell$, 
then $\gl(\tE)=1$.
Let $E\=\tE\cup\set{(x,x):x\in\oi}$.

For $x\in\oio$ and $n$ so large that
$A_n(x)\=(x-\eps_2(n),x+\eps_2(n))\subset(0,1)$, use
\refL{LLD} to find $w_n(x)$ and a set
$B_n(x)\subseteq A_n(x)$ with $\gl(B_n(x))\ge\gdn\gl(A_n(x))=2\gdn\eps_2(n)$
such that
\begin{equation}
  \label{s1}
\cnx{W-w_n(x)}{B_n(x)}\le1/n.
\end{equation}
If $(x,y)\in E$ and $x\neq y$, then
$(x,y)\in\tE$ so for all large $n$, $(x,y)\in
E_n=E_{\eps_2(n),\etan}$, and thus, by
\eqref{een},
\begin{equation}  \label{s2}
  \begin{split}
\int_{B_n(x)\times B_n(y)}&
|W(x',y')-W(x,y)|\dd\gl_{B_n(x)} (x')\dd\gl_{B_n(y)} (y')
\\&
\le
(2\gdn\eps_2(n))\qww
\int_{A_n(x)\times A_n(y)}
|W(x',y')-W(x,y)|\dd x'\dd y'
\\&
<\gdn\qww\etan=\xfrac1n.
  \end{split}
\raisetag\baselineskip
\end{equation}

Let $\chi$ be a Banach limit, \ie, a multiplicative linear
functional on $\ell^\infty$ such that
$\chi((a_n)_1^\infty)=\lim_\ntoo a_n$ if the limit exists.
Now define
\begin{equation}\label{W'}
  W'(x,y)\=
  \begin{cases}
	\chi\bigpar{(w_n(x))_n}, & y=x, \\
W(x,y),& y\neq x.
  \end{cases}
\end{equation}
Note that $W'$ is a graphon and a version of
$W$. (Lebesgue measurability is immediate, since the diagonal is
a null set.)

Assume for the rest of the proof that $\xxm\in(0,1)^m$ with $(x_i,x_j)\in E$
for all $i$ and $j$. For sufficiently large $n$,
\eqref{s1} holds for all $x_i$ and \eqref{s2}
holds for all pairs $(x_i,x_j)$ with $x_i\neq
x_j$. Thus, if $x_i=x_j$, by \eqref{s1},
\begin{equation}
  \label{s3}
\cnx{W-w_n(x_i)}{B_n(x_i)\times B_n(x_j)} \le 1/n,
\end{equation}
and if $x_i\neq x_j$, by \eqref{s2}, since the cut
norm is at most the $L^1$ norm, 
\begin{equation}
  \label{s4}
\cnx{W-W(x_i,x_j)}{B_n(x_i)\times B_n(x_j)} 
\le 1/n.
\end{equation}
For notational convenience, we define the constants
\begin{equation}\label{em1}
\wijn\=
  \begin{cases}
	w_n(x_i), & 
         x_i=x_j, \\
W(x_i,x_j),& x_i\neq x_j,
  \end{cases}
\end{equation}
and let $\bin\=B_n(x_i)$.
Thus, \eqref{s3} and \eqref{s4} say that for all $i,j\in[m]$, 
\begin{equation}
  \label{s5}
\cnx{W-\wijn}{\bin\times\bjn} \le 1/n.
\end{equation}

We extend the definition of $\phiw$ in \eqref{phiw} to families
$(W\ij)_{1\le i<j\le m}$ of functions and write
\begin{equation*} 
\Phi[(W\ij)](\yym)\=
\Phi\bigpar{(W\ij(x_i,x_j))_{i<j}}.
\end{equation*}

A standard argument shows that, for $|W\ij|\le1$, say, for all
$i$ and $j$, and
any sets $B_1,\dots,B_m\subseteq\oi$ with positive measures, the mapping
\begin{multline*}
  (W\ij)\mapsto \Phix[(W\ij);\bbm]
\\
\=\int_{\bbmx}\Phi[(W\ij)](\yym)
\dd\gl_{B_1}(y_1)\dotsm \dd\gl_{B_m}(y_m)
\end{multline*}
is Lipschitz in cut norm, in each variable
separately; by linearity it 
suffices to consider the case when 
$\Phi$ is a monomial (and thus $\phiw=\psifw$ for
some graph $F$), and this result then is explicit in
\cite[Proof of Lemma 2.2]{BR}, see also  \cite{BCLSV1}.
Thus, by \eqref{s5}, recalling that each $\wijn$ here is
a constant,
\begin{equation}\label{s7}
\Phix[(W);\bbnm]  -\Phi((\wijn)_{i<j})=O(1/n).
\end{equation}
On the other hand,
$\Phi[(W)](\yym)=\phiw(\yym)=\gax=0$ \aex, by
assumtion, and thus 
$\Phix[(W);\bbnm]=0$. Consequently, \eqref{s7} yields
\begin{equation}\label{er}
\Phi((\wijn)_{i<j})=O(1/n).
\end{equation}

Apply the Banach limit $\chi$ to \eqref{er}.
With $z\ij\=\chi((\wijn)_{i<j})$ we
obtain, recalling that $\Phi$ is a polynomial,
\begin{equation}\label{s9}
\Phi((z\ij)_{i<j})=0.
\end{equation}
If $x_i\neq x_j$, then, by \eqref{em1},
$\wijn=W(x_i,x_j)$ for all $n$, and thus
$z\ij=W(x_i,x_j)=W'(x_i,x_j)$, see \eqref{W'}.
If $x_i=x_j$, then \eqref{em1} shows
that $\wijn=w_n(x_i)$, and thus, using \eqref{W'},
$z\ij=\chi((w_n(x_i))_n)=W'(x_i,x_j)$. 
Consequently, $z\ij=W'(x_i,x_j)$ for all $(i,j)$, and
\eqref{s9} can be written
$\Phi_{W'}(\xxm)=0$, as asserted.
(In order to avoid any worry of edge effects, we have considered
$x_i\in\oio$ only. For completeness, we, trivially, may
define $W'(0,0)\=W'(1,1)\=W'(\frac12,\frac12)$.)
\end{proof}

Finally, we mention another technical problem, which might be of
interest in some applications:
\begin{problem}
  The version $W'$ in \refL{LD} is Lebesgue
  measurable. Can $W'$ always be chosen to be Borel measurable?
\end{problem}
(The construction in the proof above, using a Banach limit, does not seem to
guarantee Borel measurability.)

\newcommand\AAP{\emph{Adv. Appl. Probab.} }
\newcommand\JAP{\emph{J. Appl. Probab.} }
\newcommand\JAMS{\emph{J. \AMS} }
\newcommand\MAMS{\emph{Memoirs \AMS} }
\newcommand\PAMS{\emph{Proc. \AMS} }
\newcommand\TAMS{\emph{Trans. \AMS} }
\newcommand\AnnMS{\emph{Ann. Math. Statist.} }
\newcommand\AnnPr{\emph{Ann. Probab.} }
\newcommand\CPC{\emph{Combin. Probab. Comput.} }
\newcommand\JMAA{\emph{J. Math. Anal. Appl.} }
\newcommand\RSA{\emph{Random Struct. Alg.} }
\newcommand\ZW{\emph{Z. Wahrsch. Verw. Gebiete} }
\newcommand\DMTCS{\jour{Discr. Math. Theor. Comput. Sci.} }

\newcommand\AMS{Amer. Math. Soc.}
\newcommand\Springer{Springer}
\newcommand\Wiley{Wiley}

\newcommand\vol{\textbf}
\newcommand\jour{\emph}
\newcommand\book{\emph}
\newcommand\inbook{\emph}
\def\no#1#2,{\unskip#2, no. #1,} 
\newcommand\toappear{\unskip, to appear}

\newcommand\webcite[1]{
\texttt{\def~{{\tiny$\sim$}}#1}\hfill\hfill}
\newcommand\webcitesvante{\webcite{http://www.math.uu.se/~svante/papers/}}
\newcommand\arxiv[1]{\webcite{arXiv:#1.}}

\def\nobibitem#1\par{}

\end{document}